\documentclass[]{article}
\usepackage{setspace}
\usepackage{pdfpages}
\usepackage{float}
\usepackage[english]{babel}
\usepackage[utf8]{inputenc}
\usepackage{babel,csquotes,xpatch}
\usepackage[margin=1in]{geometry}
\usepackage{fancyhdr}
\usepackage{amsmath}
\usepackage{amssymb}
\usepackage{centernot}
\usepackage{mathtools}
\usepackage{amssymb}
\usepackage{tikz-cd}
\usepackage{tkz-euclide}
\usepackage[toc,page]{appendix}
\usetikzlibrary{decorations.markings}
\usepackage{esint}
\usepackage{url}
\usepackage{listings}
\usetikzlibrary{intersections}
\usetikzlibrary{calc}
\usepackage{amsthm}
\usepackage{dsfont}
\usepackage{makecell}
\usepackage{color} %
\definecolor{mygreen}{RGB}{28,172,0} %
\definecolor{mylilas}{RGB}{170,55,241}
\usepackage{upquote} %
\usepackage{eurosym} %
\usepackage{grffile}
\usepackage{subcaption}%
\usepackage{hyperref}
\hypersetup{
    colorlinks=true,
    citecolor=blue,
    linkcolor=blue,
    filecolor=magenta,      
    urlcolor=blue,
    pdftitle={Overleaf Example},
}
\usepackage{longtable} %
\usepackage{booktabs}  %
\usepackage[inline]{enumitem} %
\usepackage[normalem]{ulem} %
\usepackage{mathrsfs}
\usepackage[linesnumbered, ruled, lined, shortend]{algorithm2e}
\usepackage[style=trad-abbrv, backend=biber, doi = true, url = false, isbn = false, giveninits=true]{biblatex}
\date{September 15, 2025}
\AtEveryBibitem{\clearfield{issn}}

\usepackage{titlesec}

\titleformat*{\section}{\large\bfseries}
\titleformat*{\subsection}{\bfseries}

\allowdisplaybreaks

\title{Constructive proofs for some semilinear PDEs on $H^2(e^{|x|^2/4},\mathbb{R}^d)$}
\author{Maxime Breden\thanks{CMAP, CNRS, \'Ecole polytechnique, Institut Polytechnique de
Paris, 91120 Palaiseau, France. \href{mailto:maxime.breden@polytechnique.edu}{maxime.breden@polytechnique.edu}, \href{mailto:hugo.chu@polytechnique.edu}{hugo.chu@polytechnique.edu}} $\quad$ Hugo Chu\footnotemark[1]$\;$\thanks{Department of Mathematics, Imperial College London, London SW7 2AZ, United Kingdom.} }

\theoremstyle{definition}
\newtheorem{thm}{Theorem}
\newtheorem{defn}[thm]{Definition}
\newtheorem{notation}[thm]{Notation}
\newtheorem{cor}[thm]{Corollary}
\newtheorem{rmk}[thm]{Remark}
\newtheorem{lem}[thm]{Lemma}

\newtheorem{ex}[thm]{Example}

\renewcommand{\d}{\mathrm{d}}
\newcommand{\diff}[2]{\frac{\mathrm{d}#1}{\mathrm{d}#2}}
\newcommand{\pdiff}[2]{\frac{\partial#1}{\partial#2}}

\newcommand{\R}{\mathbb{R}}
\newcommand{\OP}{P_{\scriptscriptstyle\infty}}
\newcommand{\ih}{h_{\scriptscriptstyle\infty}}
\newcommand{\iH}{\mathcal{H}_{\scriptscriptstyle\infty}}
\newcommand{\Sp}{\mathbb{S}}

\newcommand{\cL}{\mathcal{L}}

\renewcommand{\d}{\mathrm{d}}

\newcommand{\N}{\mathbb{N}}

\newcommand{\bu}{\bar{u}}

\begin{document}

\maketitle

\begin{abstract}
We develop computer-assisted tools to study semilinear equations of the form
\begin{equation*}
-\Delta u -\frac{x}{2}\cdot \nabla{u}= f(x,u,\nabla u) ,\quad x\in\mathbb{R}^d.
\end{equation*}
Such equations appear naturally in several contexts, and in particular when looking for self-similar solutions of parabolic PDEs. We develop a general methodology, allowing us not only to prove the existence of solutions, but also to describe them very precisely. We introduce a spectral approach based on an eigenbasis of $\mathcal{L}:= -\Delta -\frac{x}{2}\cdot \nabla$ in spherical coordinates, together with a quadrature rule allowing to deal with nonlinearities, in order to get accurate approximate solutions. We then use a Newton--Kantorovich argument, in an appropriate weighted Sobolev space, to prove the existence of a nearby exact solution.
We apply our approach to nonlinear heat equations, to nonlinear Schr\"odinger equations and to a generalised viscous Burgers equation, and obtain both radial and non-radial self-similar profiles.
\end{abstract}

\begin{center}
{\bf \small Keywords} \\ \vspace{.05cm}
{ \small Self-similar solutions $\cdot$ Elliptic PDEs $\cdot$ Computer-assisted proofs $\cdot$ Weighted Sobolev spaces $\cdot$ Unbounded domains}
\end{center}

\begin{center}
{\bf \small Mathematics Subject Classification (2020)}  \\ \vspace{.05cm}
{\small  35C06 $\cdot$ 35J61 $\cdot$  65G20 $\cdot$ 65N15 } 
\end{center}

\section{Introduction}

Partial differential equations (PDEs) are fundamental in the modelling of many physical and biological systems, and the computation of their solutions is often well understood on bounded domains.
However, certain PDEs can lose some key features such as particular solutions or conservation laws when restricted to a bounded domain. Furthermore, many problems, such as the ones studied in this work, are more natural to consider on an unbounded domain, either because their structure leads to simpler proofs or because they automatically give rise to localised solutions whose decay is intrinsically prescribed by the equation.

Among those, diffusion problems (both linear and nonlinear), many of which are natural to consider on $\R^d$, have a simple treatment on bounded domains based on standard Sobolev spaces but which fails to extend directly on unbounded domains. Many alternative techniques have been developed or applied to the study of diffusions on $\R^d$ such as variational formulations~\cite{Escobedo1987VariationalEquation, Kavian1994Self-similarEquation, Weissler1986RapidlyEquations}, Lyapunov (including energy/entropy) techniques and shooting methods for radial solutions~\cite{Brezis1986AAbsorption, Peletier1986On0}. In particular, one popular approach is to recover some estimates which hold in the bounded case such as the Poincaré inequality by working on weighted Sobolev spaces instead.

Slightly more recently, a new family of tools has emerged, namely computer-assisted proofs (CAPs), which are complementary to the classical techniques. One of the benefits of CAPs is that the existence results they provide always come hand in hand with a quantitative description of the obtained solutions. Another advantage is that CAPs typically generalise in a straightforward way to systems, which is not necessarily the case for some of the aforementioned techniques.
Given a PDE, abstractly written as $F(u)=0$, and a numerically obtained approximate solution $\bu$, a common feature of many CAPs is to consider an associated Newton-like operator of the form
\begin{align}
\label{eqn:fixed_point}
T(u) = u - AF(u),
\end{align}
where $A$ is a well-chosen approximation of $DF(\bu)^{-1}$, and then to prove that $T$ is a contraction on a small neighbourhood of $\bu$, in a suitable norm~\cite{vandenBerg2015RigorousDynamics}.

Following the seminal papers of Nakao~\cite{Nakao1988AProblems} and Plum~\cite{Plum1992ExplicitProblems}, many authors contributed to the further development and popularisation of computer-assisted proofs for elliptic PDEs (see e.g.~\cite{Arioli2010Computer-AssistedEquation,Day2007ValidatedPDEs,Gomez-Serrano2019Computer-assistedSurvey,Oishi1995NumericalEquations,Takayasu2013VerifiedDomains,vandenBerg2017ValidationProblem,Wanner2017Computer-assistedModel,Yamamoto1998ATheorem,Zgliczynski2001RigorousEquation}), but the vast majority of these works can only be applied to problems on bounded domains. Until recently, the main exception was the technique developed by Plum, which takes $A$ equal to $DF(\bu)^{-1}$ and relies on eigenvalues bounds obtained via the so-called homotopy method in order to control $\Vert DF(\bu)^{-1}\Vert$. This approach can handle a variety of problems on unbounded domains, see~\cite{Wunderlich2022Computer-assistedObstacle} for a recent example and~\cite[Part II]{Nakao2019NumericalEquations} for a more general presentation. Regarding other CAP techniques on unbounded domains, but for more specific problems, we also mention~\cite{Nagatou2002VerifiedOscillators} which studies the spectrum of coupled harmonic oscillators, and~\cite{Nagatou2012EigenvalueOperators} which deals with eigenvalue exclusion for localised perturbations of a one-dimensional periodic Schr\"odinger operator. Lately, other general CAP methodologies were developed for (systems of) PDEs on unbounded domains, first specifically for radial solutions~\cite{vandenBerg2015RigorousProblem,vandenBerg2023ConstructiveRd}, and then for more general solutions in~\cite{Cadiot2024RigorousMethods,Cadiot2025StationaryExistence}. In these recent approaches, one goes back to a problem on a bounded domain, either by compactifying the domain, or by dealing separately with the solutions at infinity, or by introducing a periodised problem whose solutions approximate those of the original problem. Such techniques are well-suited to study localised solutions of equations having only bounded coefficients. 

In this work, we study elliptic PDEs of the form
\begin{align}
\label{eqn:general_elliptic_prblm}
    \cL u :=-\Delta u -\frac{x}{2}\cdot \nabla{u}= f(x,u,\nabla u) ,\quad x\in\R^d,
\end{align}
and develop new computer-assisted tools for such equations. Equations of the form~\eqref{eqn:general_elliptic_prblm} have already been studied extensively, as they naturally appear in different contexts such as the ones we present below, or when studying self-similarity in various problems (see \cite{Aguirre1990Self-SimilarProblems, Benachour2004AsymptoticEquations, Corrias2014ExistencePlane} or~\cite{Jia2014Local-in-spaceSolutions} in the case of the Navier--Stokes equations).

As a first example, consider the nonlinear heat equation
\begin{align}
\label{eqn:semiheat}
    \partial_t v = \Delta v -\varepsilon |v|^{p-1}v,  \qquad \varepsilon =\pm 1.
\end{align}
If $u$ solves~\eqref{eqn:general_elliptic_prblm} with $f(x,u,\nabla u) = f(u) = u/(p-1) -\varepsilon |u|^{p-1}u$, i.e.
\begin{equation}\label{eqn:intro-heat}
    \mathcal{L}u = \frac{u}{p-1} -\varepsilon |u|^{p-1}u,
\end{equation}
then $v(t, x)= t^{-1/(p-1)}u(x/\sqrt{t})$ is a self-similar solution of~\eqref{eqn:semiheat}. 
One typically seeks ``physical'' solutions, having Gaussian decay or as critical points of some energy functional. In the radial case, a thorough description of solutions can then be achieved via shooting or ODE methods (see~\cite{Brezis1986AAbsorption, Peletier1986On0, Yanagida1996UniquenessEquation} and references therein). Alternatively, a variational approach~\cite{Escobedo1987VariationalEquation, Weissler1986RapidlyEquations} identifies such solutions by working directly on weighted Sobolev spaces. 

As a second example, notice that, if $u(x) = e^{-|x|^2/8}\varphi(x)$ for an arbitrary function $\varphi$, then
\begin{align}\label{eqn:unknown-change}
    \cL u = e^{-|x|^2/8} \left(-\Delta \varphi + \left(\frac{|x|^2}{16}+\frac{d}{4}\right)\varphi \right),
\end{align}
therefore equation~\eqref{eqn:general_elliptic_prblm} is also connected to Schr\"odinger equations. Indeed, consider for instance the critical nonlinear Schr\"{o}dinger equation
\begin{equation}
    i\partial_t v + \Delta v +\varepsilon|v|^{4/d}v = 0, \label{eqn:nonlin-schrodinger}
\end{equation}
for $\varepsilon = \pm 1$. It is established (also via variational methods) in~\cite{Kavian1994Self-similarEquation} that~\eqref{eqn:nonlin-schrodinger} has self-similar solutions of the form
\begin{equation*}
    v(t, x) = (1+t^2)^{-d/4}\exp\left(\frac{it|x|^2}{4(1+t^2)}\right)e^{i\omega\arctan{t}}\varphi\left(\frac{\sqrt{2}x}{(1+t^2)^{1/2}}\right),
\end{equation*}
where $\omega \in \R$ and $\varphi : \R^d \longrightarrow \mathbb{C}$ solves
\begin{equation}\label{eqn:phi-schrodinger-intro}
-\Delta \varphi +\frac{\vert x \vert^2}{16}\varphi +\frac{\omega}{2} \varphi = \frac{\varepsilon}{2}|\varphi|^{4/d}\varphi.
\end{equation}
Using the aforementioned change of unknown $\varphi(x) = e^{|x|^2/8}u(x)$, we get back to equation~\eqref{eqn:general_elliptic_prblm} for $u$, with 
\begin{align*}
    f(x,u,\nabla u) = f(x,u) = \left(\frac{d}{4}-\frac{\omega}{2}\right) u +\frac{\varepsilon}{2}e^{|x|^2/(2d)}|u|^{4/d}u.
\end{align*}

In this work, we develop a computer-assisted approach based on~\eqref{eqn:fixed_point} to study both radial and non-radial solutions of~\eqref{eqn:general_elliptic_prblm}, which provides a precise and quantitative description of self-similar solutions of equations~\eqref{eqn:semiheat} and~\eqref{eqn:nonlin-schrodinger}. Defining
$$H^n(\mu) = \left\{f \in L^2(\mu) \mid \partial_{\beta} f \in L^2(\mu) \text{ for all } \beta \in \N^d, |\beta|\leq n\right\}, \qquad \mu(x) = \frac{\Gamma(d/2)}{(2\sqrt{\pi})^d} e^{|x|^2/4},$$
we will rigorously enclose solutions in $H^2(\mu)$ with respect to the norm $\|\cdot\|_{H^2(\mu)}:= \|\cL \cdot \|_{L^2(\mu)}$, where $L^2(\mu)$ is equipped with the usual inner product
$$\langle u,v \rangle = \int_{\R^d} u(x) v(x) \mu(x) \d x,$$
and where $\mathcal{L}$ will be shown to be a positive $L^2(\mu)$ self-adjoint operator. Here are examples of results that can be obtained using this approach. 

\begin{thm}
\label{thm:intro_frac}
Let $d=2$, $p=5/3$, $\varepsilon=1$ and $\bar{u}:\R^2\to\R$ be the radial function whose radial profile restricted to the interval $[0,4]$ is represented in Figure~\ref{fig:frac-sol}, and whose precise description is available at~\cite{Chu2024CodeOn}. There exists a positive radial solution $u^{\star} \in H^2(\mu)$ to Eq.~\eqref{eqn:intro-heat} such that
    $$\|u^{\star} - \bar{u}\|_{H^2(\mu)}\leq 1.6\times 10^{-23}.$$
\end{thm}

\begin{figure}[H]
    \centering
    \includegraphics[width=0.6\textwidth]{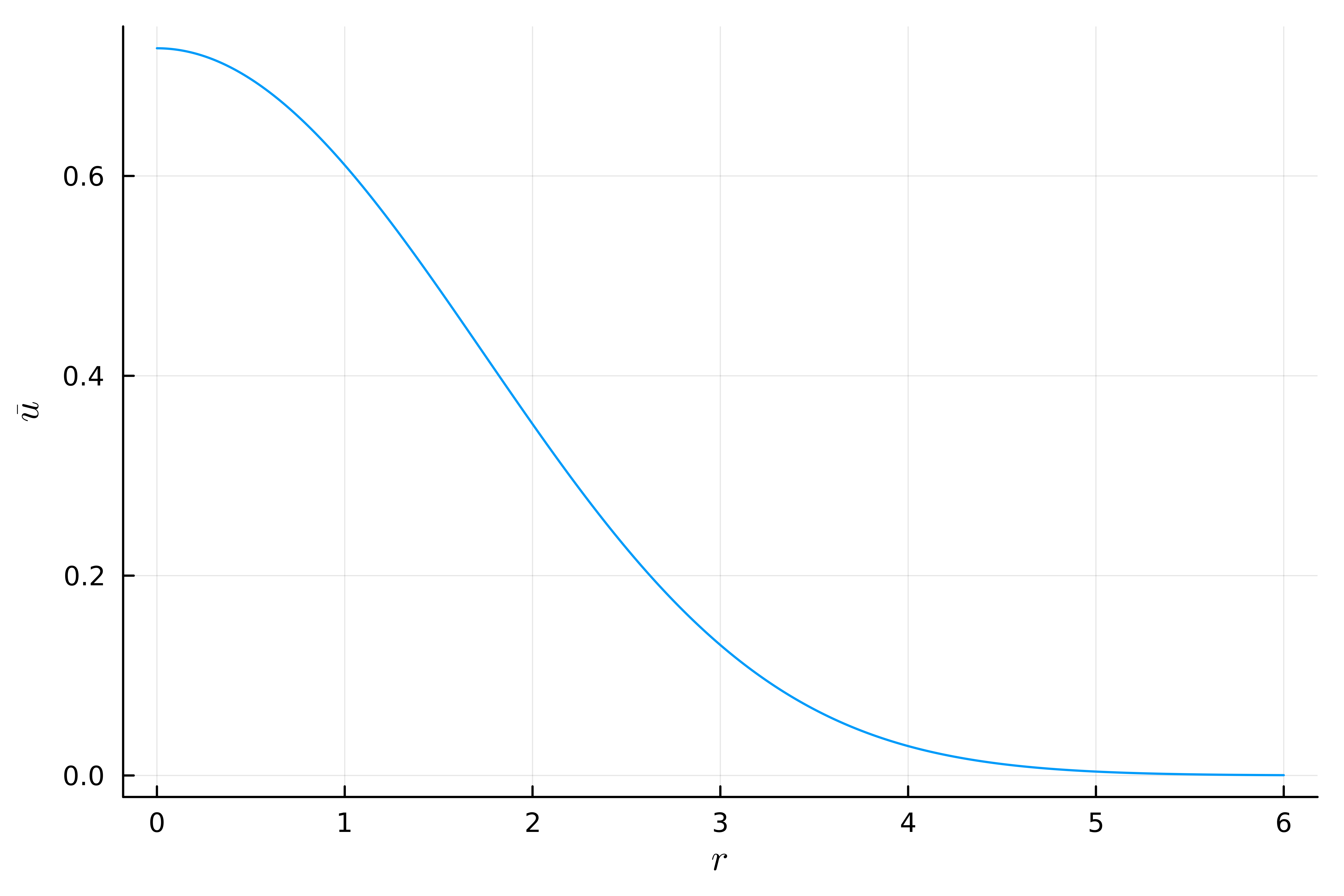}
    \caption{An approximate self-similar radial profile $\bar{u}$ to Eq.~\eqref{eqn:semiheat} with $\varepsilon = +1$, $d=2$ and $p=5/3$.\label{fig:frac-sol}}
\end{figure}

\begin{thm}
\label{thm:intro_asym}
    Let $d=2$, $\varepsilon = -1$, $\omega = -5/2$, and $\bar{\varphi}:\R^2\to\R$ be the function whose restriction to the disk $\{(x, y) \mid x^2+y^2\leq 7^2\}$ is represented in Figure~\ref{fig:asymmetric-schrodinger}, and whose precise description is available at~\cite{Chu2024CodeOn}. There exists a solution $\varphi^{\star}\in H^2(\R^2)$ to Eq.~\eqref{eqn:phi-schrodinger-intro} such that
    $$\|\varphi^{\star} - \bar{\varphi}\|_{H^2(\R^2)}\leq 3.8\times 10^{-8}.$$
\end{thm}

\begin{figure}[H]
    \centering
    \includegraphics[width=0.6\textwidth]{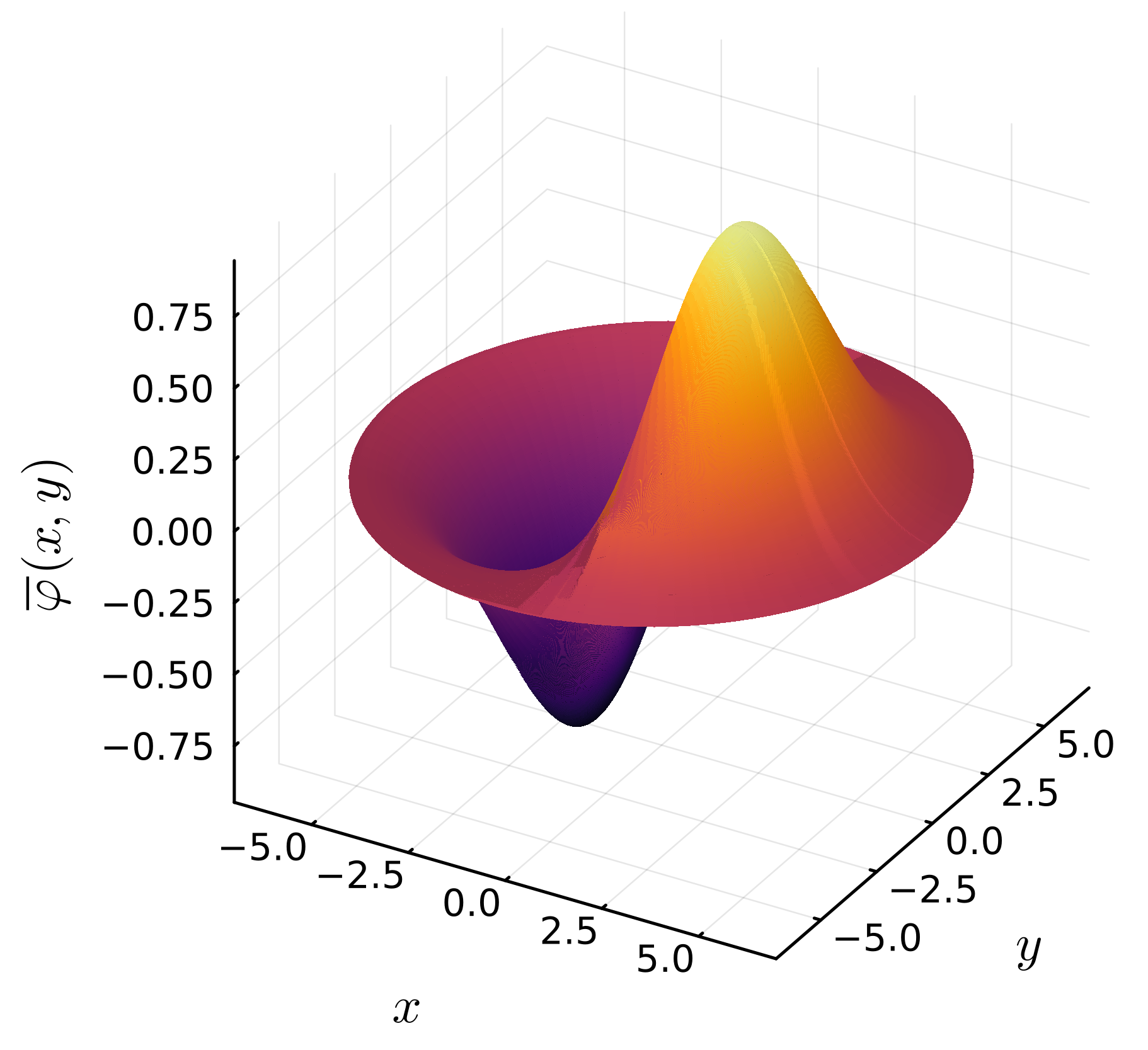}
    \caption{A non-radial numerical solution $\bar{\varphi} = e^{r^2/8} \bar{u}$ to Eq.~\eqref{eqn:phi-schrodinger-intro} with $d=2$, $\varepsilon = -1$ and $\omega = - 5/2$. \label{fig:asymmetric-schrodinger}}
\end{figure} 

More results with different self-similar profiles associated to the nonlinear heat and Schr\"odinger equations can be found in Section~\ref{sec:heat} and Section~\ref{sec:Schrodinger}.
\begin{rmk}
While the radial solutions we obtain were already known to exist thanks to previous works~\cite{Brezis1986AAbsorption, Peletier1986On0, Weissler1986RapidlyEquations}, our results also provide a precise description of these solutions. Furthermore, our methods can also deal with non-radial solutions, without relying on a perturbation argument. The only other work we are aware of where non-radial solutions are found is~\cite[Remark 3.11]{Escobedo1987VariationalEquation}, which relies on local bifurcation theory and therefore only holds for $p$ close to specific values.
\end{rmk}
\begin{rmk}
Beyond the results themselves, we believe part of the interest of this work lies in the computer-assisted tools that are developed in order to handle equations of the form~\eqref{eqn:general_elliptic_prblm}, and in the underlying ideas which may prove useful for CAPs, and more largely for numerical analysis, in a broader context. 
\begin{itemize}
    \item Our CAPs are based on a spectral expansion of the solution, which is a common approach on bounded domains, but has proven challenging on unbounded domains. This is made possible by departing from the usual paradigm which consists in choosing the basis only in terms of the leading order differential operator (i.e.~$\Delta$ in~\eqref{eqn:general_elliptic_prblm}), so that this operator becomes easy to invert by hand. Indeed, the basis we use makes $\cL:= -\left(\Delta + \frac{x}{2}\cdot \nabla\right)$ diagonal (and therefore trivial to invert), but the operator $\Delta$ alone does not have a well-behaved inverse with respect to this basis.
    \item Extracting some compactness is usually key for CAPs based on~\eqref{eqn:fixed_point}, and can be especially challenging on unbounded domains. Here, compactness is retrieved by working with appropriately weighted Sobolev spaces (see Section~\ref{sec:spacesandnorms}).
    \item We use spherical coordinate basis functions of the Hilbert space $H^2(\mu)$ which are the counterparts of the spherical harmonics in the Hermite--Laguerre family and allow us to handle our problems globally and with a unified treatment. To the best of our knowledge, this is the first occurrence of the use of these bases to solve partial differential equations, even from a purely numerical perspective. Indeed, this differs from the previous literature (both theoretical~\cite{Escobedo1987VariationalEquation} and computational~\cite{Funaro1991ApproximationFunctions})
    on this family of problems which have exhibited (multivariate) Hermite functions as an eigenbasis for $\mathcal{L}$. While at first glance, these bases might appear tricky to work with (especially in nonlinear problems), we demonstrate that they are well-suited to our problems. We show that they can be handled efficiently to both provide a spectral method to find numerical solutions to problems of the form~\eqref{eqn:general_elliptic_prblm} and then validate these solutions rigorously. On the one hand, these bases allow us to handle non-radial solutions to radially symmetric equations. On the other hand, in the case of radial solutions, these bases allow an approach based on the analysis of the operator $\mathcal{L}$ while exploiting the dimensionality reduction given by radial equations.
    \item As already mentioned, spectral expansions are common in CAPs for nonlinear equations, but only when they are associated to a Banach algebra structure which facilitates the control of nonlinearities. The basis used in this work does not enjoy this property, but we still manage to handle some nonlinear terms, thanks to a combination of rigorous quadrature and Sobolev embeddings, in a fashion which is more reminiscent of CAPs based on finite elements~\cite[Part I]{Nakao2019NumericalEquations}.
    \item Still regarding nonlinear terms, let us mention that we also handle non-polynomial nonlinearities, namely a fractional power $p$ in~\eqref{eqn:semiheat}. The existing CAP techniques based on spectral methods are more convenient to use when only polynomial nonlinearities occur, and therefore a common approach in the presence of non-polynomial terms is to first reformulate the equation in a polynomial one, typically adding several new unknowns and associated equations in the process~\cite{Henot2021OnEquations,Kepley2019ChaoticTheorem,Lessard2016AutomaticApproach}. For a different approach allowing to directly treat some non-polynomial terms without adding extra equations and unknowns, but relying heavily on a Banach algebra structure, we refer to~\cite{Breden2024Computer-assistedSensing}. One can also approximate smooth non-polynomial terms using interpolation and computable  error bounds~\cite{Arioli2006AProblem,Figueras2017RigorousApproach,vanderAalst2025PeriodicEquation}, but once again, the Banach algebra structure is then typically used to control perturbations around a given approximate solution. In the current work, we directly deal with the fractional power in a somewhat ad-hoc manner that also does not necessitate the introduction of extra equations and unknowns, and works without a Banach algebra.
\end{itemize} 
\end{rmk}

Our approach can also handle some equations featuring first-order derivatives, when $f$ in~\eqref{eqn:general_elliptic_prblm} really does depend on $\nabla u$. As an example, we also study self-similar solutions of a generalised viscous Burgers equation on $\R_+$ (see, e.g.~\cite{Aguirre1990Self-SimilarProblems, SrinivasaRao2003Self-similarDamping})
\begin{equation}\label{eqn:burger}
    \partial_t v + v^2\partial_x v = \partial_{xx} v,
\end{equation}
of the form
$$v(t, x) = t^{-1/4}u\left(\frac{x}{\sqrt{t}}\right).$$
Then, one has to find a profile $u$ which solves an equation akin to~\eqref{eqn:general_elliptic_prblm} with $f(x,u,\partial_{x} u)= u/4 - u^2 \partial_{x} u$:
\begin{align}
\label{eqn:elliptic_Burger}
    \partial_{xx} u + \frac{x}{2} \partial_{x} u + \frac{u}{4}-u^2 \partial_{x} u = 0,\quad x\in\R_+.
\end{align}

Finally, let us conclude the introduction by emphasising that there have been previous occurrences of CAPs of somewhat different nature for self-similar solutions, and that this topic has been attracting a lot of attention recently for fluid flows~\cite{Buckmaster2025SmoothFluids,Chen2025StableNumerics}. 

The remainder of the paper is organised as follows. We start by presenting the main important properties of the operator $\mathcal{L}$, together with appropriate weighted Sobolev spaces and bases of eigenfunctions in Section~\ref{sec:L}. We show how this setting allows us to recover a crucial Poincaré inequality and projection error estimates in Section~\ref{sec:spacesandnorms}, and provide some explicit constants for the Sobolev embeddings between these weighted spaces in Section~\ref{sec:embeddings}. The CAP methodology and the associated fixed-point problem where all these ingredients will be brought together is then introduced in Section~\ref{sec:fixedpoint}. Proving the existence of a fixed point requires us to explicitly evaluate (or at least upper-bound) several quantities associated to an approximate solution, and we present in Section~\ref{sec:quadrature} a quadrature scheme allowing us to compute these quantities rigorously. In Section~\ref{sec:heat}, we then focus on self-similar profiles for the nonlinear heat equation. Explicit bounds allowing to use the CAP in the case of radial solutions and with an integer exponent $p$ are derived in Section~\ref{sec:bounds}, and used in Section~\ref{sec:results_heat} to obtain our first set of results. In Section~\ref{sec:positivity} we provide a sufficient condition which we use to prove the positivity of the obtained solutions. We then show how to deal with some fractional exponents $p$ in Section~\ref{sec:heat-frac}, and with non-radial solutions in Section~\ref{sec:heat-asym}. Self-similar solutions of the nonlinear Schr\"odinger equations are treated in an analogous manner in Section~\ref{sec:Schrodinger}, and we end in Section~\ref{sec:Burger} with a generalised viscous Burgers equation. All the computations presented in this paper have been implemented in \texttt{Julia}, using the \texttt{IntervalArithmetic.jl} library 
\cite{david_p_sanders_2024_10459547} for interval arithmetic. The computer-assisted parts of the proofs can be reproduced using the code available at~\cite{Chu2024CodeOn}.

\section{Functional analytic preliminaries}
\label{sec:prelim}

\subsection{Properties of the operator $\mathcal{L}$}
\label{sec:L}

One of the main ideas of this work is to make use of the operator
$$\mathcal{L} = - \frac{x}{2}\cdot\nabla -\Delta,$$
to play the role analogous to that of the Laplacian $-\Delta$ as it is usually employed for computer-assisted proofs on a bounded domain. Indeed, consider the weighted space $L^2(\mu)$, with scalar product
\begin{align*}
    \langle u,v \rangle = \int_{\R^d} u(x) v(x) \mu(x) \d x,
\end{align*}
where $\mu(x) = e^{|x|^2/4}/Z$, $Z = 2^{d-1}\omega_{d-1}$ where $\omega_{d-1} = 2\pi^{d/2}/\Gamma(d/2)$ denotes the surface area of the $(d-1)$-dimensional sphere $\Sp^{d-1}$ (this choice of normalisation is made in order to match the usual convention in numerical analysis).
By integration by parts, we obtain the Dirichlet form
\begin{align}
\label{eq:IPP_mu}
    \langle \nabla u,\nabla v \rangle = \langle u,\mathcal{L} v \rangle, \qquad \mbox{for all $u\in H^1(\mu), v\in \mathcal{D}(\mathcal{L}) = H^2(\mu)$,}
\end{align}
which yields that $\mathcal{L}$ is a positive self-adjoint operator on $L^2(\mu)$, and can be analysed in the same ``Hilbertian'' fashion as $-\Delta$ is in a first course on elliptic PDEs, as will be seen below. We refer to~\cite{Kavian1987RemarksEquation} for more properties of $\mathcal{L}$, and note that the above Dirichlet form structure is also often exploited in the analysis of the generators of reversible diffusions (see for instance~\cite{Pavliotis2014StochasticApplications}, we explore the numerical analysis of more general equations of the form $V'(x)\partial_xu -\partial_{xx}u = f(x,u)$ on the weighted Sobolev space $H^1(e^{-V})$ in~\cite{Breden2025SolutionsSpaces}).

In the previous literature on the operator $\mathcal{L}$~\cite{Escobedo1987VariationalEquation, Funaro1991ApproximationFunctions}, its eigenbasis is usually given in cartesian coordinates in terms of Hermite functions
\begin{equation}
    \mathcal{L}\Psi_{\beta} = -\left(\frac{x}{2}\cdot\nabla + \Delta \right)\Psi_{\beta} = \frac{d+|\beta|}{2}\Psi_{\beta}, \qquad \mbox{where } \Psi_{\beta}(x) = e^{-|x|^2/4}\prod_{i=1}^{d-1}H_{\beta_i}(x_i/2), \qquad \beta \in \N^d, \label{eqn:hermite-basis}
\end{equation}
where $H_n$ denotes the Hermite polynomial of degree $n$. This basis might be suitable for some problems, but since the equations treated in this work are radially symmetric, we will instead use bases in spherical coordinates (except for $d=1$).

\begin{defn}[Spherical coordinate basis]\label{def:spherical-basis}
    For $d\in\{2,3\}$, define the following orthogonal basis of $L^2(\mu)$.
    \begin{itemize}
        \item For $d=2$, $m \in \N$, $l\in \mathbb{Z}$, $x_1 = r\cos\vartheta$, $x_2 = r\sin \vartheta$,
        $$\psi_{l,m}(r, \vartheta) = e^{-r^2/4}(r/2)^{|l|} L_m^{(|l|)}(r^2/4)\begin{cases}
            \cos l\vartheta \qquad &\mbox{if $l \geq 0$}\\
            -\sin l \vartheta \qquad &\mbox{if $l<0$,}
        \end{cases}$$
        where $L^{(\alpha)}_m$ denotes the generalised Laguerre polynomial of degree $m$ and parameter $\alpha$ (see~\cite[p. 775]{Abramowitz1970HandbookSeries}).
        \item For $d=3$, $m \in \N$, $k\in \mathbb{Z}$, $l\in \N_{\geq |k|}$, $x_1 = r\sin\vartheta\cos\phi$, $x_2 = r\sin \vartheta\sin\phi$, $x_3 = r\cos\vartheta$
        $$\psi_{k,l,m}(r, \vartheta,\phi) = e^{-r^2/4}(r/2)^l L_m^{(l+1/2)}(r^2/4)P^k_l(\cos\vartheta)\begin{cases}
            \cos k\phi \qquad &\mbox{if $k \geq 0$}\\
            -\sin k \phi \qquad &\mbox{if $k<0$,}
        \end{cases}$$
        where $P^k_l$ denotes the associated Legendre polynomial of order $k$ and degree $l$ (see~\cite[Chapter 8]{Abramowitz1970HandbookSeries}).
    \end{itemize}
\end{defn}
\begin{rmk}
    For $d=2$, the $\psi_{l,m}$'s are weighted and reformulated versions of the so-called Hermite polynomials of complex variables~\cite{Ito1952ComplexIntegral}. See~\cite{Chen2014OnOperators} for the derivations of some of these relations.
\end{rmk}
\begin{lem}
    The families $\{\psi_{l,m}\}$ and $\{\psi_{k,l,m}\}$ from Definition~\ref{def:spherical-basis} are eigenfunctions of $\mathcal{L}$ for $d=2$ and $d=3$ respectively.
    \begin{itemize}
        \item For $d=2$, $m \in \N$, $l\in \mathbb{Z}$,
        $$\mathcal{L}\psi_{l,m} = -\left(\frac{r}{2}\pdiff{\psi_{l,m}}{r}+\frac{1}{r}\pdiff{}{r}\left(r\pdiff{\psi_{l,m}}{r}\right)+\frac{1}{r^2}\pdiff{^2\psi_{l,m}}{\vartheta^2}\right) = \left(\frac{2+|l|}{2}+m\right)\psi_{l,m}.$$
        \item For $d=3$, $m \in \N$, $k\in \mathbb{Z}$, $l\in \N_{\geq |k|}$,
        \begin{align*}
        \mathcal{L}\psi_{k,l,m} &= -\left(\frac{r}{2}\pdiff{\psi_{k,l,m}}{r}+\frac{1}{r^2}\pdiff{}{r}\left(r^2\pdiff{\psi_{k,l,m}}{r}\right)+\frac{1}{r^2\sin{\vartheta}}\pdiff{}{\vartheta}\sin{\vartheta}\left(\pdiff{\psi_{k,l,m}}{\vartheta}\right)+\frac{1}{r^2\sin^2{\vartheta}}\pdiff{^2\psi_{k,l,m}}{\phi^2}\right)\\
        &= \left(\frac{3+l}{2}+m\right)\psi_{k,l,m}.
        \end{align*}
    \end{itemize}
\end{lem}
\noindent The construction of the bases given in Definition~\ref{def:spherical-basis} hinges on the following lemma.
\begin{lem}
    Let $f$ be a hyperspherical function on $\Sp^{d-1}$ of order $l\in \N$, i.e.
    $$\Delta_{\mathbb{S}^{d-1}} f = -l(l+d-2)f,$$
    where $\Delta_{\mathbb{S}^{d-1}}$ denotes the Laplace-Beltrami operator on $\mathbb{S}^{d-1}$. Then $e^{-r^2/4}(r/2)^lL_m^{(l+d/2-1)}(r^2/4)f$
    is an eigenfunction of $\mathcal{L}$ with eigenvalue $(d+l)/2+m$ on $\R^d$.
\end{lem}
\begin{proof}
    This follows from separation of the radial variable and the spherical variables and by identifying the Laguerre differential equation after the change of variable $z = r^2/4$.
\end{proof}

In order to obtain the bases of Definition~\ref{def:spherical-basis}, it then suffices to have orthogonal bases of the eigenspaces of $\Delta_{\mathbb{S}^{d-1}}$, and this construction can also be used in higher dimensions (e.g.~by using~\cite{Higuchi1987Symmetric1} for the construction of an eigenbasis of $\Delta_{\Sp^{d-1}}$).
Note that as a consequence, this directly generalises the known corresponding bases for the harmonic Schr\"{o}dinger operator and for the generator of the radially symmetric Ornstein--Uhlenbeck process. Furthermore, the linear relationship between bivariate Hermite and Laguerre polynomials observed in~\cite{Chen2014OnOperators} for $d=2$ (and the one known for $d=1$) can be generalised by the above pattern.

\begin{cor}[Radial eigenbasis]
\label{cor:radial_basis}
For $d \in \N$, let $\alpha  = d/2 -1$ and let $\psi_n(r) := L^{(\alpha)}_n(r^2/4)e^{-r^2/4}$, then
    $$\mathcal{L}\psi_n = -\left(\frac{r}{2}+\frac{d-1}{r}\right)\pdiff{\psi_n}{r} -\pdiff{^2\psi_n}{r^2}= \left(\frac{d}{2}+n\right)\psi_n. $$
\end{cor}

Note that with these bases, we can handle radial equations globally without treating either the behaviour at infinity or at the origin separately (see for instance~\cite{vandenBerg2023ConstructiveRd} for an example of CAP where the behaviours at the origin and at infinity are handled separately).

\subsection{Spaces, norms and Poincaré inequality}
\label{sec:spacesandnorms}

One of the main challenges of computer-assisted proofs on unbounded domains is that the embedding $H^1(\Omega) \hookrightarrow L^2(\Omega)$ is not compact anymore for $\Omega = \R^d$, as no Poincaré inequality even holds. Observe that in our setting with weighted spaces, the eigenbasis of $\mathcal{L}$ together with~\eqref{eq:IPP_mu} immediately implies a genuine Poincaré inequality

\begin{equation}
    \left\|u\right\|_{L^2(\mu)}\leq\sqrt{\frac{2}{d}\langle \nabla u, \nabla u \rangle} = \sqrt{\frac{2}{d}}\left\|\nabla u\right\|_{L^2(\mu)} \qquad \mbox{for all $u \in H^1(\mu)$} ,\label{eqn:Poincineq}
\end{equation}
since $d/2$ is the smallest eigenvalue of $\mathcal{L}$. We can therefore consider on $H^1(\mu)$ the scalar product $\langle\cdot, \cdot \rangle_{H^1(\mu)} = \langle \nabla \cdot, \nabla \cdot\rangle $ and its associated norm, which is equivalent to the usual $H^1(\mu)$-norm. Similarly, on the space $H^2(\mu)$ we consider the scalar product $\langle\cdot, \cdot \rangle_{H^2(\mu)} = \langle \cL \cdot, \cL \cdot\rangle $ and associated norm, which is also equivalent to the usual $H^2(\mu)$-norm (see Lemma~\ref{lem:phiH2}).

Notice that these choices of norms yield, in a similar fashion to~\eqref{eqn:Poincineq},
\begin{equation}
\label{eqn:Poincineq2}
    \Vert u\Vert_{H^1(\mu)} \leq \sqrt{\frac{2}{d}}\Vert u\Vert_{H^2(\mu)} \qquad \mbox{for all $u \in H^2(\mu)$}.
\end{equation}

\begin{notation}\label{not:1}
    We denote by $\|u\|_{\infty}$ the essential supremum of $|u|$, i.e.~for a measure $\nu$ on $\Omega$,
    $$\|u\|_{\infty} := \inf\left\{M \mid |u(x)|\leq M \textrm{ for $\nu$-a.e. $x\in \Omega$}\right\}.$$
    Our notation does not depend on $\nu$, as all the measures we will be using are equivalent (i.e.~they have the same null sets). Finally, we adopt the usual convention that $u$ is said to be continuous if it has a continuous representative.
\end{notation}

Another important feature of our setup is that that $\mathcal{L}^{-1}:L^2(\mu) \to H^1(\mu)$ is well-defined (say via the Lax--Milgram theorem~\cite[Corollary V.8]{Brezis2011FunctionalEquations}, using~\eqref{eq:IPP_mu} and~\eqref{eqn:Poincineq}) and compact. More quantitatively, if $P_n$ denotes the orthogonal projection in $L^2(\mu)$
onto $\mathrm{Span}\{\psi \mid \mathcal{L}\psi = \lambda \psi \textrm{ with } \lambda \leq d/2 +n\}$, \eqref{eq:IPP_mu} yields the following error estimates
\begin{align}
\label{eq:projection_error}
    \|\mathcal{L}^{-1}(f - P_n f)\|_{H^1(\mu)} \leq \sqrt{\frac{2}{d+2n+1}}\|f - P_n f\|_{L^2(\mu)} \qquad \mbox{for all $f \in L^2(\mu)$},
\end{align}
and
\begin{align}
\label{eq:projection_error_L2H2}
    \left\Vert f - P_n f\right\Vert_{L^2(\mu)} \leq \frac{2}{d+2n+1}\left\Vert f - P_n f\right\Vert_{H^2(\mu)} \qquad \mbox{for all $f \in H^2(\mu)$},
\end{align}
which are critical for our CAPs, and in particular for controlling the quality of the approximate inverse $A$ that will be defined in Section~\ref{sec:fixedpoint}.

\begin{rmk}
\label{rmk:Pn}
    Let us highlight two evident but useful properties of the orthogonal projector $P_n$. First, $P_n$ commutes with $\mathcal{L}$ (and $\mathcal{L}^{-1}$). Second $P_n$ is also an orthogonal projection with respect to the $H^1(\mu)$ and $H^2(\mu)$ scalar products. In particular, $P_n(L^2(\mu)) = P_n(H^1(\mu)) = P_n(H^2(\mu))$.
\end{rmk}

\subsection{Sobolev embeddings}
\label{sec:embeddings}

To control the nonlinearities in our setting, Sobolev embeddings play a crucial role. We start by recalling some $H^1(\mu)$-embeddings proved in~\cite{ Escobedo1987VariationalEquation,Kurtz1983WeightedProblems, Weissler1986RapidlyEquations}.

\begin{thm}
    We have the following Sobolev embeddings:
    \begin{itemize}
        \item[(i)]\label{item:holder1} If $d=1$, $u \in H^1(\mu) \implies e^{x^2/8}u \in\mathcal{C}^{0}(\R)$. %
        \item[(ii)] If $d=2$, $H^1(\mu) \subset L^p(\mu)$ for all $p \in [2, \infty)$.
        \item[(iii)] If $d\geq 3$, $H^1(\mu) \subset L^p(\mu)$ for all $p \in [2, 2d/(d-2)]$.
    \end{itemize}
\end{thm}

While in principle, this shows the feasibility of a computer-assisted proof using this functional framework, in practice we also require the corresponding embedding constants that we derive below.

\begin{notation}\label{not:Rd}
    In this section, we make use of classical embeddings between Sobolev spaces without weights in order to derive the needed results on weighted spaces. When we include $\R^d$ in the notation, e.g.~$L^2(\R^d)$, this means we refer to the space without weight.
\end{notation}

\begin{lem}\label{lem:CartEmb}Let $\theta(x) = |x|^2/4$ and denoting $\|u\|_{H^1(e^{\theta})}:=\|\nabla u\|_{L^2(e^{\theta})}$, then for all $u \in H^1(e^{\theta})$
        \begin{equation}
        \label{eq:CartEmb}
        \|u\|_{L^p(e^{\theta})} \leq \|e^{(1/2-1/p)\theta}u\|_{L^p(e^{\theta})}\leq C(d, p) \|u\|^{a}_{L^2(e^{\theta})}\|u\|^{1-a}_{H^1(e^{\theta})},
        \end{equation}
    where $a = 1+ d (1/p-1/2)$ and:
    \begin{itemize}
    \item for $d = 1$, $p \in [2, \infty]$ and $C(1, p) = 2^{1/2-1/p}$;
    \item for $d = 2$, $p\in [2, \infty)$ and, writing $p = 2(k+r)$ with $k \in \N^{*}=\N\backslash\{0\}$ and $r\in [0,1)$,
        $$C(2, 2(k+r)) = \left(k!\right)^{1/(k+r)}(k+1)^{r/(k+r)};$$
    \item for $d\geq 3$, $p \in [2, 2d/(d-2)]$ and
    $$C(d, p) = \left(\frac{1}{d(d-2)\pi}\right)^{(1-a)/2}\left\{\frac{(d-1)!}{\Gamma(d/2)}\right\}^{1/2-1/p}.$$
    \end{itemize}
\end{lem}

\begin{proof}
    First, notice that $(1/2-1/p)\theta \geq 0$ since $p\geq 2$, which yields the first inequality in~\eqref{eq:CartEmb}. 
    Then, we note that for all $u\in H^1(e^{\theta})$ and $\varphi = e^{|x|^2/8}u$, \begin{equation}\label{eqn:useful-id}
        -\frac{1}{2}\int_{\R^d}\left(x\cdot \nabla \varphi\right)\varphi\d x =  \frac{d}{4}\int_{\R^d}\varphi^2 \d x,
    \end{equation}
    as already shown in the proof of~\cite[Lemma 4.17]{Escobedo1987VariationalEquation}. Indeed, for any $\varphi \in\mathcal{C}^{\infty}_c(\R)$, an integration by part yields
    $$
    \int_{\R^d}\left(x\cdot \nabla \varphi\right)\varphi\d x = -\int_{\R^d} \varphi\nabla\cdot(x\varphi)\d x = -\int_{\R^d}\varphi(x\cdot\nabla\varphi)\d x -d\int_{\R^d}\varphi^2\d x,$$
    which gives~\eqref{eqn:useful-id}, and this identity still holds for $\varphi = e^{|x|^2/8}u$ by a density argument.
    Therefore,
    $$\|\nabla u \|_{L^2(e^{\theta})}^2=\int_{\R^d}\left|\nabla\varphi - \frac{x}{4}\varphi\right|^2\d x = \int_{\R^d}\left(\left|\nabla\varphi\right|^2-\frac{\varphi}{2}\left(x\cdot \nabla\varphi\right)+\frac{|x|^2}{16}\varphi^2\right)\d x = \int_{\R^d}\left[\left|\nabla\varphi\right|^2 +\varphi^2\left(\frac{d}{4}+\frac{|x|^2}{16}\right)\right]\d x$$
    which yields
    \begin{equation}\label{eqn:Lions-ineq}
    \|\nabla \varphi\|_{L^2(\R^d)}\leq \|\nabla u \|_{L^2(e^{\theta})}.
    \end{equation}
    \begin{itemize}
        \item For $d=1$, from the proof of~\cite[Theorem VIII.7]{Brezis2011FunctionalEquations}, we have that
        $$\|\varphi\|_{\infty} \leq 2^{1/2} \|\varphi\|_{L^2(\R)}^{1/2}\|\varphi'\|_{L^2(\R)}^{1/2}.$$
        Therefore,
        \begin{equation}\label{eqn:infty-estimate}
        \|e^{\theta/2}u\|_{\infty}\leq 2^{1/2} \|u\|_{L^2(e^{\theta})}^{1/2}\|u\|_{H^1(e^{\theta})}^{1/2},
         \end{equation}
        which is the announced inequality for $p=\infty$. For $2\leq p <\infty$, Hölder's inequality yields 
        \begin{align*}
            \|e^{\theta(1/2-1/p)}u\|_{L^p(e^{\theta})} &\leq \left(\|e^{\theta/2}u\|_{\infty}^{p-2}\|u\|^2_{L^2(e^{\theta})}\right)^{1/p} \\
            & = \|e^{\theta/2}u\|_{\infty}^{1-2/p}\|u\|^{2/p}_{L^2(e^{\theta})} \\
            & \leq 2^{1/2-1/p}\|u\|^{1/2+1/p}_{L^2(e^{\theta})}\|u\|^{1/2-1/p}_{H^1(e^{\theta})}.
        \end{align*}
    \item Now, for $d=2$, we are in the ``limiting case'',
    so from~\cite[Corollary IX.11]{Brezis2011FunctionalEquations}, for $k\in \mathbb{N}^{*}$
    $$\|\varphi\|_{L^{2k}(\mathbb{R}^2)} \leq k^{1/k}\|\varphi\|_{L^{2(k-1)}(\mathbb{R}^2)}^{(k-1)/k}\|\nabla \varphi\|_{L^2(\mathbb{R}^2)}^{1/k},$$
    and by induction on $k \in \mathbb{N}^{*}$
    $$\|\varphi\|_{L^{2k}(\R^2)} \leq (k!)^{1/k} \|\varphi\|_{L^2(\R^2)}^{1/k} \|\nabla \varphi\|_{L^2(\R^2)}^{1-1/k}. $$
    Now, by interpolation, for $2k\leq 2(k+r)\leq 2(k+1)$, we get
    $$\|\varphi\|_{L^{2(k+r)}(\R^2)} \leq (k!)^{1/(k+r)}(k+1)^{r/(k+r)}\|\varphi\|_{L^2(\R^2)}^{1/(k+r)}
    \|\nabla \varphi\|^{1-1/(k+r)}_{L^2(\R^2)},$$
    and thus
    \begin{align*}
        \|e^{\theta(k+r-1)/(2(k+r))}u\|_{L^{2(k+r)}(e^{\theta})}  = \|\varphi\|_{L^{2(k+r)}(\R^2)} \leq \left(k!\right)^{1/(k+r)}(k+1)^{r/(k+r)}\|u\|_{L^2(e^{\theta})}^{1/(k+r)}
    \|u\|^{1-1/(k+r)}_{H^1(e^{\theta})}.
    \end{align*}
    \item For $d\geq 3$, proceed similarly and use~\cite{Talenti1976BestInequality} (see also~\cite[Lemma 4.17]{Escobedo1987VariationalEquation}). \qedhere
    \end{itemize}
\end{proof}

\begin{cor}\label{cor:nonlin}
    For all $u \in H^1(\mu)$,
    $$\|u^p\|_{L^2(\mu)}\leq \|e^{(p-1)|x|^2/8}u^p\|_{L^2(\mu)} \leq c(d, p)\|u\|^p_{H^1(\mu)},$$
    where
    \begin{itemize}
        \item for $d = 1$, $p \in [1, \infty)$, $c(1, p) = 2^{(5p-3)/4}$;
        \item for $d = 2$, $p= k +r \in [1, \infty)$ with $k \in \N^*$ and $r\in [0,1)$, $c(2, k+r) = (2\sqrt{\pi})^{k+r-1}k!(k+1)^r$;
    \item for $d\geq 3$,  $p \in [1, d/(d-2)]$ and 
    $$c(d, p) = \left(\frac{2}{d}\right)^{p/2} \left\{\left(\frac{2}{d-2}\right)^{d/4}\frac{\Gamma(d)^{1/2}}{\Gamma(d/2)}\right\}^{p-1}.$$
    \end{itemize}
\end{cor}
\begin{proof}
    Noting that $\mu = e^\theta/Z$ where $Z = 2^{d-1}\omega_{d-1}$ where $\omega_{d-1} = 2\pi^{d/2}/\Gamma(d/2)$ and applying Lemma~\ref{lem:CartEmb} followed by the Poincaré inequality~\eqref{eqn:Poincineq} yields
    \begin{align*}
        \|u^p\|_{L^2(\mu)} &= \|e^{(p-1)\theta/2}u^p\|_{L^2(\mu)}\\
        &=\frac{1}{\sqrt{Z}} \|e^{\theta(1/2-1/2p)}u\|^p_{L^{2p}(e^\theta)}\\
        &\leq \frac{\left(C(d,2p)\sqrt{Z}\right)^p}{\sqrt{Z}}\left(\sqrt{\frac{2}{d}}\right)^{p+d(1/2-p/2)} \|u\|^p_{H^1(\mu)} \\
        &=  \left(C(d,2p)\right)^p \left(\frac{2^d \pi^{d/2}}{\Gamma(d/2)}\right)^{\frac{p-1}{2}}\left(\frac{2}{d}\right)^{p/2+d(1/4-p/4)} \|u\|^p_{H^1(\mu)} .
    \end{align*}
    One then just has to check that
    \begin{equation*}
        c(d, p) = \left(C(d,2p)\right)^p \left(\frac{2^d \pi^{d/2}}{\Gamma(d/2)}\right)^{\frac{p-1}{2}}\left(\frac{2}{d}\right)^{p/2+d(1/4-p/4)}. \qedhere
    \end{equation*}
\end{proof}

While the $H^1(\mu)\to L^2(\mu)$ embeddings will prove useful for some intermediate estimates, our CAP will mainly be conducted in $H^2(\mu)$.  We derive the required embeddings with explicit constants below.

\begin{lem}
\label{lem:phiH2}
    For all $f\in L^2(\mu)$, there exists a unique $u \in H^2(\mu)$ such that $\cL u = f$, and $\cL^{-1} : L^2(\mu) \rightarrow H^2(\mu) $ is bounded, so that $\|\cdot\|_{H^2(\mu)} :=  \|\mathcal{L}\cdot\|_{L^2(\mu)}$ indeed induces the $H^2(\mu)$ topology. Furthermore, if $u\in H^2(\mu)$ and $\varphi = e^{|x|^2/8}u$, then $\varphi\in H^2(\R^d)$ and
    \begin{equation}\label{eqn:Hessian-ineq}
        \|D^2 \varphi\|_{L^2(\R^d)} = \|\Delta\varphi\|_{L^2(\R^d)} \leq\sqrt{Z}\|u\|_{H^2(\mu)},
    \end{equation}
    where 
    $\|D^2 \varphi\|_{L^2(\R^d)} = \left(\sum_{i,j=1}^d\|\partial_{x_ix_j}\varphi\|_{L^2(\R)}^2\right)^{1/2}$ denotes the Frobenius norm of the Hessian $D^2\varphi$ of $\varphi$.
\end{lem}
\begin{proof}
    The qualitative version of this lemma is the content of~\cite[Lemma 2.1.(vii)]{Kavian1987RemarksEquation} (The additional fact that $\varphi\in H^2(\R^d)$ follows from~\cite[Theorem IX.25]{Brezis2011FunctionalEquations}). Given the bounds obtained in the proof of~\cite[Lemma 2.1.(vii)]{Kavian1987RemarksEquation}, for all $u\in H^2(\mu)$ the estimates below hold by density for $\varphi = e^{|x|^2/8}u$. Using~\eqref{eqn:unknown-change}, we have that
    \begin{align*}
        \int_{\R^d}\left(\cL u\right)^2e^{|x|^2/4}\d x &= \int_{\R^d}\left[\left|\Delta \varphi\right|^2+\left(\frac{d}{4}+\frac{|x|^2}{16}\right)\varphi^2\right]\d x -\int_{\R^d}\varphi\Delta\varphi\left(\frac{d}{2}+\frac{|x|^2}{8}\right)\d x.
    \end{align*}
    Now, by integration by parts, the last term simplifies to
    \begin{align*}
         -\int_{\R^d}\varphi\Delta\varphi\left(\frac{d}{2}+\frac{|x|^2}{8}\right)\d x &= \int_{\R^d}\left|\nabla \varphi\right|^2\left(\frac{d}{2}+\frac{|x|^2}{8}\right)\d x+\int_{\R^d}\varphi\left(\frac{x}{4}\cdot\nabla\varphi\right)\d x\\
         &=\int_{\R^d}\left|\nabla \varphi\right|^2\left(\frac{d}{2}+\frac{|x|^2}{8}\right)\d x-\frac{d}{8}\int_{\R^d}\varphi^2\d x\qquad \mbox{using~\eqref{eqn:useful-id}.}
    \end{align*}

    Adding back this term, we get
    \begin{equation}\label{eqn:full-H2-ineq}
        \int_{\R^d}\left(\cL u\right)^2e^{|x|^2/4}\d x \geq \int_{\R^d}\left|\Delta \varphi\right|^2\d x +\frac{d}{2}\int_{\R^d}\left|\nabla \varphi\right|^2\d x+\frac{d}{8}\int_{\R^d}\varphi^2\d x,
    \end{equation}
    and in particular
    $$\|D^2\varphi\|_{L^2(\R^d)} = \|\Delta\varphi\|_{L^2(\R^d)}\leq \|\cL u\|_{L^2(e^{|x|^2/4})} \leq \sqrt{Z}\|u\|_{H^2(\mu)},$$
    where the first equality follows from~\cite[Lemma 2.5, Corollary 2.13]{Spina2005EquazioniLp}, since $\varphi\in H^2(\mathbb{R}^d)$.
\end{proof}
Applying a similar argument as in Lemma~\ref{lem:CartEmb}, we then obtain the following embeddings.
\begin{cor}\label{cor:H2-emb}
    We have the following Sobolev embeddings:
    \begin{itemize}
        \item[(i)]\label{item:holder2} If $d\in\{1,2,3\}$, $u \in H^2(\mu) \implies e^{|x|^2/8}u \in L^{\infty}(\R^d)$. 
        \item[(ii)] If $d=4$, $H^2(\mu) \subset L^p(\mu)$ for all $p \in [2, \infty)$.
        \item[(iii)] If $d\geq 5$, $H^2(\mu) \subset L^p(\mu)$ for all $p \in [2, 2d/(d-4)]$.
    \end{itemize}
\end{cor}
\begin{proof}
    The proof follows from the application of the Sobolev embeddings given by~\cite[Corollary IX.13]{Brezis2011FunctionalEquations} on $\varphi = e^{|x|^2/8}u \in H^2(\R^d)$.
\end{proof}
\begin{rmk}
    Note that these embeddings are sufficient to cover the range of exponent $p$ for the semilinear heat equation for which rapidly decaying self-similar solutions are known to exist~\cite{Escobedo1987VariationalEquation}.
\end{rmk}
We now provide an explicit embedding constant for the case that we will need for our examples, namely (i) in Corollary~\ref{cor:H2-emb}, for $d\in\{2,3\}$ (for $d=1$, one can use the $H^1(\mu)$-embedding from Corollary~\ref{cor:nonlin} together with~\eqref{eqn:Poincineq2}).
\begin{lem}\label{lem:infty-bound}
    Let $d\in \{2, 3\}$, then there exists $C(d)>0$ such that for all $u\in H^2(\mu)$

    $$\|e^{|x|^2/8}u\|_{\infty}\leq C(d) \|u\|_{H^2(\mu)},$$
    where $C(d)$ can be chosen as
    $$C(d) = \sqrt{Z}\left(\frac{2}{d}C_0 +\sqrt{\frac{2}{d}}C_1 +C_2\right),$$
    with
    $$\begin{tabular}{c c c c}
        $C_0 = 0.56419$, & $C_1 = 0.79789$,& $C_2 = 0.23033$, & if $d=2;$\\
        $C_0 = 0.52319$, & $C_1 = 1.0228$, & $C_2 = 0.37467$, & if $d=3.$
    \end{tabular}$$
\end{lem}

\begin{proof}
Letting once again $\varphi = e^{|x|^2/8}u$, we have that $\varphi\in H^2(\R^d)$ by Lemma~\ref{lem:phiH2}. We can thus use the embedding from $H^2(\R^d)$ to $L^\infty(\R^d)$, with the explicit constants provided in~\cite[Example 6.12.(a)]{Nakao2019NumericalEquations}:
\begin{align*}
    \Vert \varphi\Vert_{\infty} \leq C_0 \Vert \varphi\Vert_{L^2(\R^d)} + C_1 \Vert \nabla\varphi\Vert_{L^2(\R^d)} + C_2 \Vert D^2\varphi\Vert_{L^2(\R^d)}.
\end{align*}
Then, noticing that $\Vert u\Vert_{L^2(\mu)} \leq \frac{2}{d} \Vert u\Vert_{H^2(\mu)}$ (by a direct computation or by combining~\eqref{eqn:Poincineq} and~\eqref{eqn:Poincineq2}), we get
\begin{align*}
    \Vert \varphi\Vert_{L^2(\R^d)} = \sqrt{Z}\Vert u\Vert_{L^2(\mu)} \leq \sqrt{Z}\frac{2}{d}\Vert u\Vert_{H^2(\mu)}.
\end{align*}
Similarly, starting from~\eqref{eqn:Lions-ineq} and using once more~\eqref{eqn:Poincineq2}, we get
\begin{align*}
    \Vert \nabla \varphi\Vert_{L^2(\R^d)} \leq \sqrt{Z}\Vert u\Vert_{H^1(\mu)} \leq \sqrt{Z}\sqrt{\frac{2}{d}}\Vert u\Vert_{H^2(\mu)},
\end{align*}
and~\eqref{eqn:Hessian-ineq} provides the last missing inequality.
\end{proof}

\subsection{The fixed-point theorem}
\label{sec:fixedpoint}

As is common in the CAP literature, our approach will rely on a fixed-point theorem allowing for the \emph{a posteriori validation} of an appropriate numerical solution $\bar{u}$, that is, a proof of the existence of a solution $u^{\star}$ to
\begin{equation}\label{eqn:gen-poisson}
    \cL u = f(x, u, \nabla u),
\end{equation}
in an explicit neighbourhood of $\bar{u}$ in a Banach space $\mathcal{X}$. This is achieved by constructing a well-chosen map $T : \mathcal{X}\to \mathcal{X}$ with solutions of Eq.~\eqref{eqn:gen-poisson} as its fixed points and which is contracting in a neighbourhood of $\bar{u}\in \mathcal{X}$. In order to actually prove that $T$ is a contraction, we will use the following statements (see~\cite{Nakao2019NumericalEquations,vandenBerg2021SpontaneousFlow,Yamamoto1998ATheorem} and the references therein for variants).

\begin{thm}\label{thm:fixed-point}
    Let $(\mathcal{X}, \|\cdot\|_{\mathcal{X}})$ be a Banach space, let $T:\mathcal{X}\to \mathcal{X}$ be a $\mathcal{C}^1$-map and let $\bar{u}\in \mathcal{X}$ be such that there exists a constant $Y\geq0$ and an increasing continuous function $\mathcal{Z}:\mathbb{R}_+\to \R_+$ such that
    \begin{align}
        \|T(\bar{u})-\bar{u}\|_{\mathcal{X}}&\leq Y,\label{eqn:Ycond}\\
        \|DT(u)\|_{\mathcal{X}, \mathcal{X}}&\leq \mathcal{Z}(\|u-\bar{u}\|_{\mathcal{X}})\qquad \mbox{for all $u\in \mathcal{X}$.}\label{eqn:Zcond}
    \end{align}
    If there exists $\delta>0$ such that
    \begin{align}
        Y + \int_0^{\delta}\mathcal{Z}(s)\d s &< \delta, \label{eqn:condinto}\\
        \mathcal{Z}(\delta)<1, \label{eqn:condcontract}
    \end{align}
    then $T$ has a unique fixed point $u^{\star}$ in $\bar{B}(\bar{u},\delta)$.
\end{thm}

\begin{proof}
    See for instance~\cite[Theorem 2.1]{Breden2019RigorousPaths}
\end{proof}

{If $T$ is polynomial in $u$ of degree $p$, we reformulate condition~\eqref{eqn:Zcond} into bounding $D^kT(\bar{u})$ for $k \in \{1, \ldots, p\}$.}

\begin{cor}\label{cor:fixed-point}
    Let $T:\mathcal{X}\to \mathcal{X}$ be a polynomial of degree $p$ and let $\bar{u}\in \mathcal{X}$ be such that there exist $Y, Z_1, \ldots, Z_p\geq 0$ such that
    \begin{align}
        \|T(\bar{u})-\bar{u}\|_{\mathcal{X}}&\leq Y,\\
        \|D^kT(\bar{u})h^k\|_{\mathcal{X}} &\leq Z_k\|h\|_{\mathcal{X}}^k \qquad \mbox{for all $h\in\mathcal{X}$ and $k\in \{1, \ldots, p\}$.}
    \end{align}
    Introducing the radii polynomials
    \begin{align*}
        P(\delta) &= Y - \delta +\sum_{k=1}^p\frac{Z_k}{k!}\delta^k,\\
        Q(\delta) &= -1 +\sum_{k=1}^p\frac{Z_k}{(k-1)!}\delta^{k-1},
    \end{align*}
   if there exists $\delta>0$ such that $P(\delta)<0$, then denoting $\underline{\delta}$ the smallest positive root of $P$ and $\bar{\delta}$ the unique positive root of $Q$, for all $\delta \in (\underline{\delta}, \bar{\delta})$ $T$ has a unique fixed point $u^{\star} \in \bar{B}(\bar{u},\delta) \subset \mathcal{X}$.
\end{cor}

\begin{proof}
    See for instance~\cite[Corollary 4.4]{Breden2019RigorousPaths}.
\end{proof}

We now explain how to construct suitable operator $T$ in order to study solutions of equation~\eqref{eqn:gen-poisson}.
Provided $g:u \mapsto f(\cdot,u,\nabla u)$ maps $H^2(\mu)$ to $L^2(\mu)$, our general problem
$$\cL u = f(x,u,\nabla u)$$
can be reformulated as finding a zero of the map $F :H^2(\mu) \to H^2(\mu)$, where
$$F(u) = u - \mathcal{L}^{-1}f(x,u,\nabla u).$$
Note that if $Dg(\bar{u})$ in fact maps $ H^2(\mu)$ to $H^1(\mu)$, then $DF(\bu)$ is a compact perturbation of the identity. In particular, and this is a key aspect common to many such computer-assisted proofs, we can expect to find an accurate approximate inverse of $DF(\bu)$ which is also a compact perturbation of the identity. This motivates the introduction of the linear map $A :H^2(\mu)\to H^2(\mu)$ such that $P_n A P_n = A_n$ is an approximate inverse of $P_nDF(\bar{u})P_n$ computed numerically, and where $P_n A(I-P_n) =(I-P_n)AP_n = 0 $ and $(I-P_n)A(I-P_n) = I-P_n$. That is, using the block decomposition associated with the orthogonal projection $P_n$, we have
\begin{align*}
    A = 
    \begin{pmatrix}
       A_n & 0 \\
       0 & I
    \end{pmatrix},
\end{align*}
and we expect such an $A$ to be accurate approximate inverse of $DF(\bu)$ if $n$ is large enough. This statement will be made quantitative in Section~\ref{sec:boundsZ1}, by leveraging the explicit compactness estimate~\eqref{eq:projection_error_L2H2}, which allows us to obtain a computable upper bound for $\left\Vert I-ADF(\bu)\right\Vert_{H^2(\mu),H^2(\mu)}$ .

We therefore consider the quasi-Newton operator $T : H^2(\mu) \longrightarrow H^2(\mu)$ with
$$T(u) =  u - AF(u),$$ 
which we expect to be contracting near $\bu$.
Provided $A$ is invertible (this will be an automatic consequence of verifying assumption~\eqref{eqn:condcontract}), zeros of $F$ are indeed in one-to-one correspondence with fixed points of $T$. In Sections~\ref{sec:heat} to~\ref{sec:Burger}, we derive suitable estimates $Y$ and $\mathcal{Z}$ (or $Z_k$) allowing us to apply Theorem~\ref{thm:fixed-point} or Corollary~\ref{cor:fixed-point} to this fixed-point operator $T$ for the different equations mentioned in the introduction.

\begin{rmk}
    One could choose to perform a computer-assisted proof in $H^1(\mu)$ instead of $H^2(\mu)$, but the class of treatable problems would be more restricted and this would not directly yield $L^{\infty}$-estimates on the solution in dimension $d \in \{2, 3\}$.
\end{rmk}

\section{Handling nonlinearities via quadrature rules}
\label{sec:quadrature}

To derive the bounds $Y$ and $\mathcal{Z}$ needed for our CAP (See Theorem~\ref{thm:fixed-point}), we will need to rigorously compute quantities like $\|\bar{u}^p\|_{L^2(\mu)}$. In the more classical case of Fourier series, $\bar{u}$ is typically a finite Fourier series (i.e.~a trigonometric polynomial), and $\bar{u}^p$ is then still a finite Fourier series (provided $p$ is an integer), whose integral or norm is therefore straightforward to compute exactly. In contrast, in our setting, whenever $u\neq 0$ and $v\neq 0$ belong to $P_n(H^2(\mu))$, the product $uv$ does not belong to $P_{n'}(H^2(\mu))$ for any $n'\in\N$. However, we can still compute exactly $\|uv\|_{L^2(\mu)}$ (or $\|\bar{u}^p\|_{L^2(\mu)}$), and we explain how this can be done in this section. 

For clarity, we focus here on the case of the radial equation 
$$\mathcal{L}u =u/(p-1) -\varepsilon u^p,\quad \varepsilon=\pm 1,\ p\in \N,$$
giving rise to self-similar solutions of the semilinear heat equation, but all the constructions can be adapted to the other examples. We denote by $\{\hat{\psi}_m\}_{m=0}^{\infty}$ the normalisation of the orthogonal basis $\{\psi_m\}_{m = 0}^{\infty}$ (see Corollary~\ref{cor:radial_basis}) with respect the $L^2(\mu)$-inner product $\langle \cdot, \cdot\rangle$, i.e.

$$\hat{\psi}_m(r) = \sqrt{\frac{m!}{\Gamma(m+\alpha+1)}}L^{(\alpha)}_m(r^2/4)e^{-r^2/4}, \qquad \alpha = \frac{d}{2}-1.$$

We will solve our problem with respect to this basis and look for our numerical approximation $\bar{u}$ in $\mathrm{Span}\{\hat{\psi}_m\}_{m=0}^n$ (which is the set of radial functions in $P_n(H^2(\mu))$). We can then write $\bar{u}$ in terms of its ``Fourier'' coefficients:
\begin{equation}\label{eqn:ubar-id}
    \bar{u} = \sum_{m = 0}^n \bar{u}_m \hat{\psi}_m \qquad \bar{\mathbf{u}} = (\bar{u}_0, \ldots,\bar{u}_n).
\end{equation}
Since our computer-assisted proof will mostly rely on a Hilbertian approach, one of the main challenges is to be able to compute many projections and integrals (eventually rigorously) on unbounded domains. Such endeavour might appear naive given the lack of simple error control on Gauss--Laguerre quadrature and the notorious difficulties of working with this family of bases~\cite{Funaro1990ComputationalApproximations}. Furthermore, we would ideally like to solve (via gradient descent or Newton's method) the Galerkin problem
\begin{equation}\label{eqn:galerkin}
    \left\langle \mathcal{L}\bar{u} -\bar{u}/(p-1) +\varepsilon\bar{u}^p, \hat{\psi}_j\right\rangle = 0 \qquad \mbox{for all $j \in \{0, \ldots, n\}$},
\end{equation}
without further approximating the problem. We show that this can in fact be achieved relatively efficiently for a nonlinearity of moderate degree by simply adapting the Gauss--Laguerre quadrature. Consider more generally the problem of evaluating (eventually rigorously) integrals of the form
$$\int_{\R^d}\prod_{k=1}^m f_k(x) \mu(x)\d x,$$
for $f_k \in \mathrm{Span}\{\hat{\psi}_m\}_{m=0}^n$. For instance, $\|\bar{u}^p\|_{L^2(\mu)}$ can be written in this form, but also the coefficients of $P_n(\bar{u}^p)$ given by $\langle \bar{u}^p,\hat{\psi}_m \rangle$ for $m=0,\ldots,n$. Since we are in the radial setting, this boils down to evaluating integrals of the form
\begin{equation}
\label{eq:intergrals}
    \int_{0}^{\infty}\prod_{k=1}^m f_k(r) (r/2)^{d-1}e^{r^2/4}\d r,
\end{equation}
with $f_k(r) = p_k(r^2/4) e^{-r^2/4}$ for some polynomial $p_k$ of degree at most $n$. Note that all calculations can be thought to be done with respect to the variable $z = r^2/4$.
\begin{rmk}
    When we have to deal with non-radial solutions, we also use spherical coordinates, and the projections and integrals with respect to the spherical variables only involve sines and cosines and thus do not require a quadrature. This is another advantage of using spherical coordinates instead of cartesian ones (i.e., of using the bases of Definition~\ref{def:spherical-basis} instead of~\eqref{eqn:hermite-basis}), for which nontrivial quadratures would be required in every direction.%
\end{rmk}
We now show how to exactly compute integrals of the form~\eqref{eq:intergrals}.
\begin{lem}\label{lem:quad}
    Let $p_1, \ldots, p_m$ be polynomials of degree at most $n$ and let $N\in\mathbb{N}$ be such that $2N-1 \geq mn$. Then, for $\alpha>-1$ and $\beta>0$
    \begin{equation}
    \int_0^{\infty}z^{\alpha}e^{-\beta z}\prod_{k=1}^m p_k(z)\d z = \beta^{-\alpha-1}\sum_{i=1}^N W_i\prod_{k=1}^mp_k\left(\frac{z_i}{\beta}\right), \label{eqn:quad}
    \end{equation}
    where $z_i$ denotes the $i$\textsuperscript{th} root of $L^{(\alpha)}_N$ and
$$W_i = \frac{\Gamma(N+\alpha+1)z_i}{N!(N+1)^2[L^{(\alpha)}_{N+1}(z_i)]^2} \qquad i = 1, \ldots, N.$$
\end{lem}
\begin{proof}
This follows from using the chain rule for the change of variables $z \mapsto z/\beta$ and the $N$\textsuperscript{th} Gauss--Laguerre quadrature (see~\cite[Formula 4.1]{Funaro1990ComputationalApproximations}) which is exact for polynomials of degree $2N-1$ or less.
\end{proof}

In our implementation, $p_k$ will usually be written as
$$p_k = \sum_{j=0}^n a^k_j L^{(\alpha)}_j/\|\psi_j\|_{L^2(\mu)}.$$
Thus, the above quadrature rule would be written as
$$\int_0^{\infty}z^{\alpha}e^{-\beta z}\prod_{k=1}^m p_k(z)\d z = \beta^{-\alpha-1}\sum_{i=1}^N W_i\prod_{k=1}^m\sum_{j=0}^n a^k_j L^{(\alpha)}_j\left(\frac{z_i}{\beta}\right)= \beta^{-\alpha-1}\sum_{i=1}^N W_i\prod_{k=1}^m (V\mathbf{a}^k)_i,$$
where $\mathbf{a}^k = (a^k_j)_{0\leq j\leq n}$ and $V=(V_{ij})=(L^{(\alpha)}_j(z_i/\beta)/\|\psi_j\|_{L^2(\mu)})$ is a (pseudo)-Vandermonde matrix. However, this matrix has very large entries so that~\eqref{eqn:quad} as written needs to be computed in high precision arithmetic. When~\eqref{eqn:quad} needs to be evaluated fast and accurately, it proves better to use a regularised Vandermonde matrix $\bar{V} = \mathrm{Diag}(\beta^{-\alpha-1}W)^{1/m}V$ 
so that
    $$\int_0^{\infty}z^{\alpha}e^{-\beta z}\prod_{k=1}^m p_k(z)\d z = \sum_{i=1}^N \prod_{k=1}^m (\bar{V}\mathbf{a}^k)_i.$$
In practice, $\bar{V}$ can be precomputed once (using high precision if needed but then possibly stored in lower precision) and after this is done both the rigorous and non-rigorous computations can be carried out in a manageable amount of time. 

\begin{rmk}
    In practice, $\bar{V}$ proves much better behaved than $V$ and easier to work with. It could be interesting to further investigate the properties of $\bar{V}$.
\end{rmk}

\begin{ex}In both the rigorous and non-rigorous parts, one needs to compute the matrix
    $$G_{ij} = \langle \hat{\psi}_i, \bar{u}^{p-1}\hat{\psi}_j\rangle$$
    which can be written as $G = \bar{V}^{T}\mathrm{Diag}(\bar{V} \bar{\mathbf{u}})^{p-1}\bar{V}$ for $m =p+1$ and $\beta = p$ (and $\bar{\mathbf{u}}$ as in~\eqref{eqn:ubar-id}).
\end{ex}

Note that here the nonlinearity in the Galerkin problem is not approximated by its so-called \emph{pseudospectral} projection (usually obtained by collocation) found in the previous related literature~\cite{Ben-yuGuo2000HermiteEquations,Funaro1991ApproximationFunctions}.

Like the validation procedure, the quadrature is implemented rigorously in \texttt{Julia} using the package \texttt{IntervalArithmetic.jl}~\cite{david_p_sanders_2024_10459547}. Our implementation also relies on approximating the roots of $L^{(\alpha)}_N$ with high precision, using initial guesses given by \texttt{FastGaussQuadrature.jl}~\cite{Townsend2014FastGaussQuadrature.jl} and refining them using Newton's method in \texttt{BigFloat} arithmetic. The roots are then enclosed rigorously by combining the Fundamental Theorem of Algebra and the Intermediate Value Theorem (recall that the roots of orthogonal polynomials are always simple and real). $V$ is then computed by recursion using the three-term recurrence relation for orthogonal polynomials. Since some of the $z_i$ have large magnitude, this calculation quickly leads to large error accumulation, hence the need to enclose the $z_i$ with high precision. While this quadrature might appear expensive, in our experience for problems with $N\leq 2000$, the enclosure of the roots together with the computation of $\bar{V}$ only took a few minutes on a standard laptop.

Putting quadrature aside, the overall computational cost of the computer-assisted proofs presented in this paper is very similar to that of more standard proofs using Fourier series: the main limitations are memory constraints and matrix-matrix multiplications in interval arithmetic. The cost of computing the quadrature, i.e., the matrix $\bar{V}$, only becomes noticeable if $N$ has to be taken somewhat large. Indeed, increasing $N$ requires increasing the precision for the calculation of the quadrature rule, and more specifically for rigorously evaluating the polynomials $L_j^{(\alpha)}$, inevitably increasing computational time. Among all the computer-assisted proofs presented in this paper, only for the one of Theorem~\ref{thm:Burgers} did the quadrature have a predominant cost.

\begin{rmk}\label{rmk:basis}
   Hermite and Laguerre bases are far less popular than Fourier and Chebyshev for computer-assisted proofs, or in numerical analysis for this matter. One reason is the lack of simple and global error control on interpolation and quadrature. However, polynomials are well-controlled in the bulk, and the tails of $\hat{\psi}_m$ can be controlled via the following estimate of Szeg\"{o}~\cite[formula 22.14.13]{Abramowitz1970HandbookSeries}
    \begin{equation}\label{eqn:tail-estimate}
        |\hat{\psi}_m(r)| \leq \hat{\psi}_m(0) e^{-r^2/8} = \frac{1}{\Gamma(\alpha+1)}\sqrt{\frac{\Gamma(m+\alpha+1)}{m!}}e^{-r^2/8},
    \end{equation}
    if $\alpha = d/2-1\geq 0$, i.e., if $d\geq 2$. In particular, this estimate in theory allows for integration with rigorous error control, say by bluntly integrating in the bulk via interval arithmetic. Since our method will only rely on integration, it might handle more general nonlinearities. Note that this estimate immediately provides $L^{\infty}$-bounds on finite linear combinations of the $\hat{\psi}_m$'s.
\end{rmk}

\section{The semilinear heat equation}
\label{sec:heat}

In this section, we are interested in the semilinear heat equation
\begin{equation}
    \partial_t v = \Delta v -\varepsilon |v|^{p-1}v \qquad \mbox{on $\R^d$},\label{eqn:semiheat2}
\end{equation}
with $p>1$. It is well-known that if $u$ solves the nonlinear elliptic problem
\begin{equation}
    -\Delta u - \frac{x}{2}\cdot \nabla{u} -\frac{u}{p-1} +\varepsilon|u|^{p-1}u = 0 ,\label{eqn:carteq}
\end{equation}
then $v(t, x)= t^{-1/(p-1)}u(x/\sqrt{t})$ is a self-similar solution of~\eqref{eqn:semiheat2}. These solutions are relevant when studying the asymptotic behaviour of solutions to Eq.~\eqref{eqn:semiheat2}. Indeed, provided that $v(0, \cdot) \in H^1(\mu)$, they appear as (possibly unstable) ``omega-limit points''~\cite{Kavian1987RemarksEquation} of $v(t, \cdot)$ in the sense that there exists a solution $u$ to Eq.~\eqref{eqn:carteq} and $t_n\rightarrow\infty$ such that

$$ \lim_{n\to\infty} \left\|t_n^{1/(p-1)}v(t_n, \cdot) - u\left(\frac{\cdot}{\sqrt{t_n}}\right)\right\|_{\infty} = 0.$$

Let us first focus on positive radial solutions of Eq.~\eqref{eqn:carteq} (we will later show in Section~\ref{sec:positivity} that the solutions we obtain are indeed positive)
 which have been the subject of~\cite{Brezis1986AAbsorption, Peletier1986On0}, so that Eq.~\eqref{eqn:carteq} becomes
\begin{equation}
    \mathcal{L}u = -\partial_r^2 u -\left(\frac{d-1}{r}+\frac{r}{2}\right)\partial_r u = \frac{u}{p-1} -\varepsilon u^p.\label{eqn:rad-form}
\end{equation}
Furthermore, let us assume first for simplicity that $p\in \N$. An example with a non-integer power $p$ will be presented in Section~\ref{sec:heat-frac}, and non-radial solutions will be studied in Section~\ref{sec:heat-asym}.

\subsection{Suitable bounds for the validation}\label{sec:bounds}

In this section, we apply the general strategy outlined in Section~\ref{sec:fixedpoint} to Eq.~\eqref{eqn:rad-form}, and derive computable estimates $Y$ and $Z_k$ satisfying the assumptions of Corollary~\ref{cor:fixed-point}.
Since we are first looking for radial solutions, we restrict ourselves to the Hilbert space
$$\mathcal{H} = \overline{\mathrm{Span}\{\psi_m\}_{m = 0}^{\infty}}^{H^2}\subset H^2(\mu).$$
Throughout this section, we then assume $\bar{u}$ is an element of $\mathcal{H}_n:=\mathrm{Span}\{\psi_m\}_{m=0}^{n} \subset \mathcal{C}_0^{\infty}(\mathbb{R}_{+})$. In particular, $\|\bar{u}\|_{\infty}$ is finite and can be computed rigorously (or at least upper-bounded rigorously, say via the estimate~\eqref{eqn:tail-estimate}).
This quantity will appear in several of the estimates to come. 

In order to study Eq.~\eqref{eqn:rad-form}, we consider
\begin{align}
    F : \mathcal{H} &\longrightarrow \mathcal{H}, \nonumber\\
    u&\longmapsto u - \cL^{-1}\left(\frac{u}{p-1} - \varepsilon u^p\right),\label{eqn:rad-heat-F}
\end{align}
and $T$ defined as in Section~\ref{sec:fixedpoint}. Since we work in $\mathcal{H}$, the approximate inverse $A$ involved in $T$ is naturally taken to go from $\mathcal{H}$ to $\mathcal{H}$, so that $T:\mathcal{H} \to \mathcal{H}$. We obtain the following formulae for $DT^k(\bar{u})$ 
$$\begin{cases}
    DT(\bar{u})h &= h - ADF(\bar{u})h =  h - A(h - \mathcal{L}^{-1}( h/(p-1) -\varepsilon p \bar{u}^{p-1}h))\\
    D^kT(\bar{u})h^k &= -\varepsilon\dfrac{p!}{(p-k)!}A\mathcal{L}^{-1}(\bar{u}^{p-k}h^k) \qquad\qquad\qquad\qquad\qquad \qquad \qquad\qquad \qquad 1<k\leq p.
\end{cases}$$

\begin{notation}\label{not:2}
    Recall Notation~\ref{not:1} and Remark~\ref{rmk:Pn}. Our analysis will mostly be done in $\mu$-weighted spaces, and we may shorten the notation by omitting the $\mu$. Henceforth, $\|\cdot\|_{H^2}$ (resp. $\|\cdot\|_{H^1}$, $\|\cdot\|_{L^2}$) will always denote $\|\cdot\|_{H^2(\mu)}$ (resp. $\|\cdot\|_{H^1(\mu)}$, $\|\cdot\|_{L^2(\mu)}$). Furthermore, we denote by $\{\hat{\psi}_m\}_{m=0}^{\infty}$ the normalisation of the radial orthogonal basis $\{\psi_m\}_{m = 0}^{\infty}$ (introduced in Corollary~\ref{cor:radial_basis}) with respect the $L^2(\mu)$-inner product $\langle \cdot, \cdot\rangle$, and let $\lambda_m = d/2+m$ so that $\cL \hat{\psi}_m = \lambda_m \hat{\psi}_m$.
     We denote $\OP  = I-P_n$ the complementary projection to $P_n$ (in $L^2(\mu)$ or $H^2(\mu)$).
     Note that $P_n$, and therefore also $\OP$, leave $\mathcal{H}$ invariant, and we abuse notation and still denote by $P_n$ and $\OP$ their restriction to $\mathcal{H}$ so that $P_n:\mathcal{H}\to\mathcal{H}$ is the orthogonal projection onto $\mathcal{H}_n$.
    We also denote by $\iH= \OP(\mathcal{H})$ the orthogonal complement of $\mathcal{H}_n$ in $\mathcal{H}$. 
    Finally, since our implementation will be with respect to the $L^2(\mu)$-basis, in what follows, vectors in $\mathbb{R}^{n+1}$ are identified with the corresponding decomposition with respect to $\{\hat{\psi}_m\}_{m=0}^{n}$ and we aim at stating final estimates with respect to the $L^2(\mu)$-norm. 
\end{notation}

\subsubsection{The bound $Y$}

It suffices to estimate
\begin{align*}
    \|T(\bar{u}) - \bar{u}\|^2_{H^2} &= \|AF(\bar{u})\|_{H^2}^2\\
    &= \|P_n(AF(\bar{u}))\|^2_{H^2} +\|\OP (AF(\bar{u}))\|_{H^2}^2\\
    &= \|A_n(P_nF(\bar{u}))\|^2_{H^2}+\|\OP F(\bar{u})\|^2_{H^2}\\
    &= \|A_n(P_nF(\bar{u}))\|^2_{H^2}+\|\OP \mathcal{L}^{-1}(\bar{u}^p)\|^2_{H^2}\\
    &= \|A_n(P_nF(\bar{u}))\|^2_{H^2}+\left(\|\mathcal{L}^{-1}\bar{u}^p\|^2_{H^2}-\|P_n\mathcal{L}^{-1}(\bar{u}^p)\|^2_{H^2}\right).%
\end{align*}
Thus, we may choose
$$Y^2:= \|\mathcal{L}A_n(P_nF(\bar{u}))\|^2_{L^2}+\left(\|\bar{u}^p\|^2_{L^2}-\|P_n(\bar{u}^p)\|^2_{L^2}\right).$$

Note that all the terms in $Y$ can be evaluated rigorously by means of a quadrature rule (see Lemma~\ref{lem:quad}) and the output is expected to be small provided the truncation level $n$ is large enough. In particular, the first term can practically be made arbitrarily small in the non-rigorous computation of $\bar{u}$. 

\begin{rmk}
    Here and in many subsequent places, we slightly abuse notation, in the sense that the number $Y$ that we will actually compute is only an upper bound of the quantity $Y$ defined above. Indeed, we use interval arithmetic to control rounding errors, and in later estimates we also only upper-bound quantities like $\left\Vert \bu\right\Vert_\infty$ rather than trying to compute them \emph{exactly}.
\end{rmk}

\subsubsection{The bound $Z_1$}
\label{sec:boundsZ1}
    We choose to compute $Z_1$ by treating finite and infinite dimensional parts separately, and therefore take
    $$Z_1 := \left\|\left(\begin{matrix}
        Z^{11} & Z^{12}\\
        Z^{21} & Z^{22}
    \end{matrix}\right)\right\|_{2}$$
    where the $Z^{ij}$ are real numbers satisfying
    \begin{align}
    \label{eq:Z1ij}
        \begin{cases}
        \|P_n DT(\bar{u})P_n\|_{H^2, H^2} &\leq Z^{11},\\
        \|\OP DT(\bar{u})P_n\|_{H^2, H^2} &\leq Z^{21},\\
        \|P_n DT(\bar{u})\OP\|_{H^2, H^2} &\leq Z^{12},\\
        \|\OP DT(\bar{u})\OP\|_{H^2, H^2} &\leq Z^{22}.\\
        \end{cases}
    \end{align}
    Indeed, for $h\in \mathcal{H}$ with $\|h\|_{H^2} = 1$, write $h = h_n+\ih$ such that $h_n \in \mathcal{H}_n$, $\ih\in \iH$. %
    Then,
    \begin{align*}
        \|DT(\bar{u})h\|_{H^2} &= \sqrt{\|P_nDT(\bar{u})h\|_{H^2}^2+\|\OP DT(\bar{u})h\|_{H^2}^2}\\
        &\leq\sqrt{\left(\|P_nDT(\bar{u})h_n\|_{H^2}+\|P_nDT(\bar{u})\ih\|_{H^2}\right)^2+\left(\|\OP DT(\bar{u})h_n\|_{H^2}+\|\OP DT(\bar{u})\ih\|_{H^2}\right)^2}\\
        &\leq\sqrt{\left(Z^{11}\|h_n\|_{H^2}+Z^{12}\|\ih\|_{H^2}\right)^2+\left(Z^{21}\|h_n\|_{H^2}+Z^{22}\|\ih\|_{H^2}\right)^2}\\
        & = \left\Vert 
            \begin{pmatrix}
                Z^{11} & Z^{12}\\
                Z^{21} & Z^{22}
            \end{pmatrix}
            \begin{pmatrix}
                \|h_n\|_{H^2} \\
                \|\ih\|_{H^2}
            \end{pmatrix}
            \right\Vert_2 \\
        &\leq Z_1 \qquad \mbox{since $\|h_n\|_{H^2}^2 +\|\ih\|_{H^2}^2 = 1$.}
    \end{align*}
In the following subsections, we derive computable bounds $Z^{ij}$'s satisfying~\eqref{eq:Z1ij}.

\subsubsection{The bound $Z^{11}$}
    For $h_n\in \mathcal{H}_n$,
    $$P_nDT(\bar{u})h_n = P_n(I - ADF(\bar{u}))h_n = (I_n - A_n P_nDF(\bar{u})P_n)h_n$$
    by construction of $A$. Then, we get the estimate
    $$\|P_n DT(\bar{u})P_n\|_{H^2, H^2} = \|I_n - A_n P_nDF(\bar{u})P_n\|_{H^2,H^2} = \|\mathcal{L}(I_n - A_n P_nDF(\bar{u})P_n)\mathcal{L}^{-1}\|_{L^2,L^2}=: Z^{11},$$
which is expected to be small because $A_n\approx (P_nDF(\bar{u})P_n)^{-1}$. Using the isometric isomorphism between $L^2(\mu)$ and $\ell^2$, this quantity $Z^{11}$ is in fact nothing but the $2$-norm of an $(n+1)\times(n+1)$ matrix, whose entries can be evaluated rigorously. For $M \in \R^{(n+1)\times (n+1)}$, we may bound above its 2-norm with the estimate $\|M\|_2 \leq \sqrt{\|M\|_1\|M\|_{\infty}}$. While this estimate is not sharp in general, it is good enough for our purpose, as we only need to compute $\|M\|_2$ either when $M$ is very close to being zero (and we only need an estimate which is below $1$), or where $\|M\|_2$ is not the crucial estimate. 

\subsubsection{The bound $Z^{21}$}
    For $h_n \in \mathcal{H}_n$, we get that
    \begin{align*}
        \OP DT(\bar{u})h_n &= -\OP DF(\bar{u}) h_n\\
        &=\OP \mathcal{L}^{-1}( h_n/(p-1) -\varepsilon p \bar{u}^{p-1}h_n)\\
        &=-\varepsilon p\OP \mathcal{L}^{-1}(\bar{u}^{p-1}h_n),
    \end{align*}
    and thus,
    $$
    \|\OP DT(\bar{u})h_n\|_{H^2}= p\|\OP (\bar{u}^{p-1}h_n)\|_{L^2}.$$
    Then, writing $h_n \in \mathcal{H}_n$ as
    $$h_n = \sum_{m=0}^n a_m\hat{\psi}_m,$$
    we get
    $$ \OP (\bar{u}^{p-1}h_n) = \sum_{m=0}^n a_m \OP  (\bar{u}^{p-1}\hat{\psi}_m).$$
    Thus,
    \begin{align*}
        \|\OP (\bar{u}^{p-1}h_n)\|_{L^2} &\leq \sum_{m=0}^n\|a_m \OP  (\bar{u}^{p-1}\hat{\psi}_m)\|_{L^2}\qquad \mbox{by Minkowski inequality}\\
        &\leq \sum_{m=0}^n|a_m \lambda_m|\|\lambda^{-1}_m\OP  (\bar{u}^{p-1}\hat{\psi}_m)\|_{L^2}\\
        &\leq \left(\sum_{m=0}^n\lambda^{-2}_m\|\OP  (\bar{u}^{p-1}\hat{\psi}_m)\|^2_{L^2}\right)^{1/2}\left(\sum_{m=0}^n\lambda_m^2 a_m^2\right)^{1/2} \qquad \mbox{by Cauchy--Schwarz inequality}\\
        &= \left(\sum_{m=0}^n\lambda^{-2}_m\|\OP  (\bar{u}^{p-1}\hat{\psi}_m)\|^2_{L^2}\right)^{1/2}\|h_n\|_{H^2}.
    \end{align*}
    Introducing the vector $w \in \mathbb{R}^{n+1}$ with
    $$w_m := \|\OP  (\bar{u}^{p-1}\hat{\psi}_m)\|_{L^2} =  \left(\|\bar{u}^{p-1} \hat{\psi}_m\|^2_{L^2}-\|P_n(\bar{u}^{p-1} \hat{\psi}_m)\|^2_{L^2}\right)^{1/2},$$
    and identifying $w$ with an element of $P_n(L^2(\mu)) \subset L^2(\mu)$, we get
    $$\|\OP DT(\bar{u})P_n\|_{H^2, H^2} \leq p\|\mathcal{L}^{-1}w\|_{L^2} =:Z^{21}.$$
Note that $w_m$ is expected to be small, especially for small $m$, whereas for larger $m$ the $\lambda_m^{-2}$ factors also contribute to making $\|\mathcal{L}^{-1}w\|^2_{L^2} = \sum_{m=0}^n\lambda^{-2}_m w_m^2$ small.





    \subsubsection{The bound $Z^{12}$}
    Now, for $\ih \in \iH$,
            \begin{align*}
            P_n DT(\bar{u})\ih &= - P_n A DF(\bar{u}) \ih\\
            &=P_n A\mathcal{L}^{-1} \ih/(p-1) -\varepsilon pP_n A  \mathcal{L}^{-1}(\bar{u}^{p-1}\ih)\\
            &=\mathcal{L}^{-1} P_n\ih/(p-1) -\varepsilon pA_nP_n \mathcal{L}^{-1}(\bar{u}^{p-1}\ih)\\
            &= -\varepsilon p  A_n\mathcal{L}^{-1}P_n(\bar{u}^{p-1}\ih) \qquad \mbox{since $\ih \in \iH$.}%
            \end{align*}
    Then, for $m \in \{0,\ldots n\}$, we estimate
        \begin{align*}
        |\langle \bar{u}^{p-1}\ih, \hat{\psi}_m\rangle| &= |\langle \bar{u}^{p-1} \hat{\psi}_m, \ih\rangle|\\
            &=|\langle \OP( \bar{u}^{p-1} \hat{\psi}_m), \ih\rangle|\qquad \mbox{since $\ih \in \iH$}\\
            &\leq \|\OP (\bar{u}^{p-1} \hat{\psi}_m)\|_{L^2}\|\ih\|_{L^2}\qquad \mbox{by Cauchy--Schwarz inequality}\\
            &\leq \frac{\|\ih\|_{H^2}}{\lambda_{n+1}}\|\OP (\bar{u}^{p-1} \hat{\psi}_m)\|_{L^2} \qquad \mbox{by~\eqref{eq:projection_error_L2H2}}.
        \end{align*}
    Thus, reusing the vector $w$ introduced in the previous subsection,
    $$\|P_nDT(\bar{u})\OP \|_{H^2, H^2} \leq {\frac{p}{\lambda_{n+1}}\left\|\,\left|A_n \mathcal{L}^{-1}\right|w\right\|_{H^2} = \frac{p}{\lambda_{n+1}}\left\|\,\left|\mathcal{L}A_n \mathcal{L}^{-1}\right|w\right\|_{L^2}}=: Z^{12},$$
    where for $M\in \R^{(n+1)\times(n+1)}$, $|M| = (|M|_{ij}) = (|M_{ij}|)$.

    \subsubsection{The bound $Z^{22}$}
    For $\ih \in \iH$,

        $$\OP DT(\bar{u})\ih =  -\OP DF(\bar{u}) \ih=\OP \mathcal{L}^{-1}( (1/(p-1) -\varepsilon p \bar{u}^{p-1})\ih.$$
        Thus,
        \begin{align*}
            \|\OP DT(\bar{u})\ih\|_{H^2}&= \|(1/(p-1) -\varepsilon p \bar{u}^{p-1})\ih\|_{L^2}\\
            &\leq\|1/(p-1) -\varepsilon p \bar{u}^{p-1}\|_{\infty}\|\ih\|_{L^2}\qquad \mbox{by H\"{o}lder's inequality }\\
            &\leq \frac{(1/(p-1) +p \|\bar{u}\|^{p-1}_{\infty})}{\lambda_{n+1}}\|\ih\|_{H^2}\qquad \mbox{by~\eqref{eq:projection_error_L2H2}}.
        \end{align*}
Thus, we may choose

$$Z^{22} := \frac{1/(p-1) +p \|\bar{u}\|^{p-1}_{\infty}}{\lambda_{n+1}},$$
where $\|\bar{u}\|_{\infty}$ can be computed (or at least upper-bounded) using~\eqref{eqn:tail-estimate}.
Note that while $T: \mathcal{H}\rightarrow \mathcal{H}$ may not be compact, $DT(\bar{u}) :\mathcal{H}
\rightarrow \mathcal{H}$ is, provided $\bar{u}\in L^{\infty}(\R^d)$.

\subsubsection{The bounds $Z_k$ for $1<k\leq p$}

If $d\leq 2$, or $p\leq d/(d-2)$ for $d\geq 3$, we can use the $H^1$-Sobolev embeddings given by Lemma~\ref{lem:CartEmb}, which yields, for $h \in \mathcal{H}$,
    \begin{align*}
        \left\|D^kT(\bar{u})h^k\right\|_{H^2} &= \frac{p!}{(p-k)!}\left\|A\mathcal{L}^{-1}(\bar{u}^{p-k}h^k)\right\|_{H^2}\\
        &\leq \frac{p!}{(p-k)!}\|A\mathcal\|_{H^2,H^2}\|\mathcal{L}^{-1}(\bar{u}^{p-k}h^k)\|_{H^2}\\
        &\leq \frac{p!}{(p-k)!}\|\mathcal{L}A\mathcal{L}^{-1}\|_{L^2,L^2}\|\bar{u}^{p-k}h^k\|_{L^2}\\
        &\leq \frac{p!\|\bar{u}\|_{\infty}^{p-k}}{(p-k)!}\|\mathcal{L}A\mathcal{L}^{-1}\|_{L^2,L^2}\|h^k\|_{L^2} \qquad \mbox{by H\"{o}lder's inequality}\\
        &\leq \frac{p!\|\bar{u}\|_{\infty}^{p-k}}{(p-k)!}\|\mathcal{L}A\mathcal{L}^{-1}\|_{L^2,L^2}c(d, k)\|h\|_{H^1}^k\\
        &\leq \frac{p!\|\bar{u}\|_{\infty}^{p-k}}{(p-k)!}\|\mathcal{L}A\mathcal{L}^{-1}\|_{L^2,L^2}c(d, k)\lambda_0^{-k/2}\|h\|_{H^2}^k,
    \end{align*}
    where $c(d, k)$ is given by Corollary~\ref{cor:nonlin} and we have used~\eqref{eqn:Poincineq2}. Thus, we may choose
    \begin{equation}
        \label{eq:Zk1}
    Z_k := \frac{p!\|\bar{u}\|_{\infty}^{p-k}}{(p-k)!}\lambda_0^{-k/2}c(d, k)\|\mathcal{L}A\mathcal{L}^{-1}\|_{L^2,L^2},
    \end{equation}
    where 
    $$\|\mathcal{L}A\mathcal{L}^{-1}\|_{L^2,L^2}\leq \max\left\{\|\mathcal{L}A_n\mathcal{L}^{-1}\|_{L^2,L^2}, 1\right\}.$$
    Note that instead of using $\|\bar{u}\|_{\infty}$, for $k<p$, we may derive a bound only using the embedding $H^1(\mu)\hookrightarrow L^{2p}(\mu)$, but this is not expected to be as sharp.

    For $d \in \{2,3\}$, we may instead use Lemma~\ref{lem:infty-bound} relying on $H^2$-Sobolev embeddings
    \begin{align*}
        \left\|D^kT(\bar{u})h^k\right\|_{H^2} &\leq \frac{p!\|\bar{u}\|_{\infty}^{p-k}}{(p-k)!}\|\mathcal{L}A\mathcal{L}^{-1}\|_{L^2,L^2}\|h^k\|_{L^2}\\
        &\leq \frac{p!\|\bar{u}\|_{\infty}^{p-k}}{(p-k)!}\|\mathcal{L}A\mathcal{L}^{-1}\|_{L^2,L^2}\|h\|_{\infty}^{k-1}\|h\|_{L^2}\\
        &\leq \frac{p!\|\bar{u}\|_{\infty}^{p-k}}{\lambda_0(p-k)!}\|\mathcal{L}A\mathcal{L}^{-1}\|_{L^2,L^2}C(d)^{k-1}\|h\|_{H^2}^k,
    \end{align*}
where $C(d)$ is given by Lemma~\ref{lem:infty-bound}. Thus, we may also choose
    \begin{equation}
    \label{eq:Zk2}
    Z_k := \frac{p!\|\bar{u}\|_{\infty}^{p-k}}{\lambda_0(p-k)!}C(d)^{k-1}\|\mathcal{L}A\mathcal{L}^{-1}\|_{L^2,L^2}.
    \end{equation}
In practice, when $d \in \{2,3\}$, if both estimates are valid, we take for $Z_k$ the minimum between~\eqref{eq:Zk1} and~\eqref{eq:Zk2}.

\subsection{Results}
\label{sec:results_heat}

We first focus on the case $\varepsilon = -1 $ in~\eqref{eqn:rad-form}, such that the equation reads
\begin{equation}
    \mathcal{L}u = -\partial_r^2 u -\left(\frac{d-1}{r}+\frac{r}{2}\right)\partial_r u = \frac{u}{p-1} + u^p.\label{eqn:rad-form-plus}
\end{equation}
In the cases $(d, p) = (2, 3)$ and $(d, p) = (3,2)$, we find the numerical solutions plotted below with $n=500$. One could of course use fewer coefficients (i.e., take $n$ smaller), but the bound $Y$ measuring the residual error and then the error bound $\underline{\delta}$ would then become larger.

\begin{figure}[h]

\begin{subfigure}{0.5\textwidth}
\includegraphics[width=0.9\linewidth]{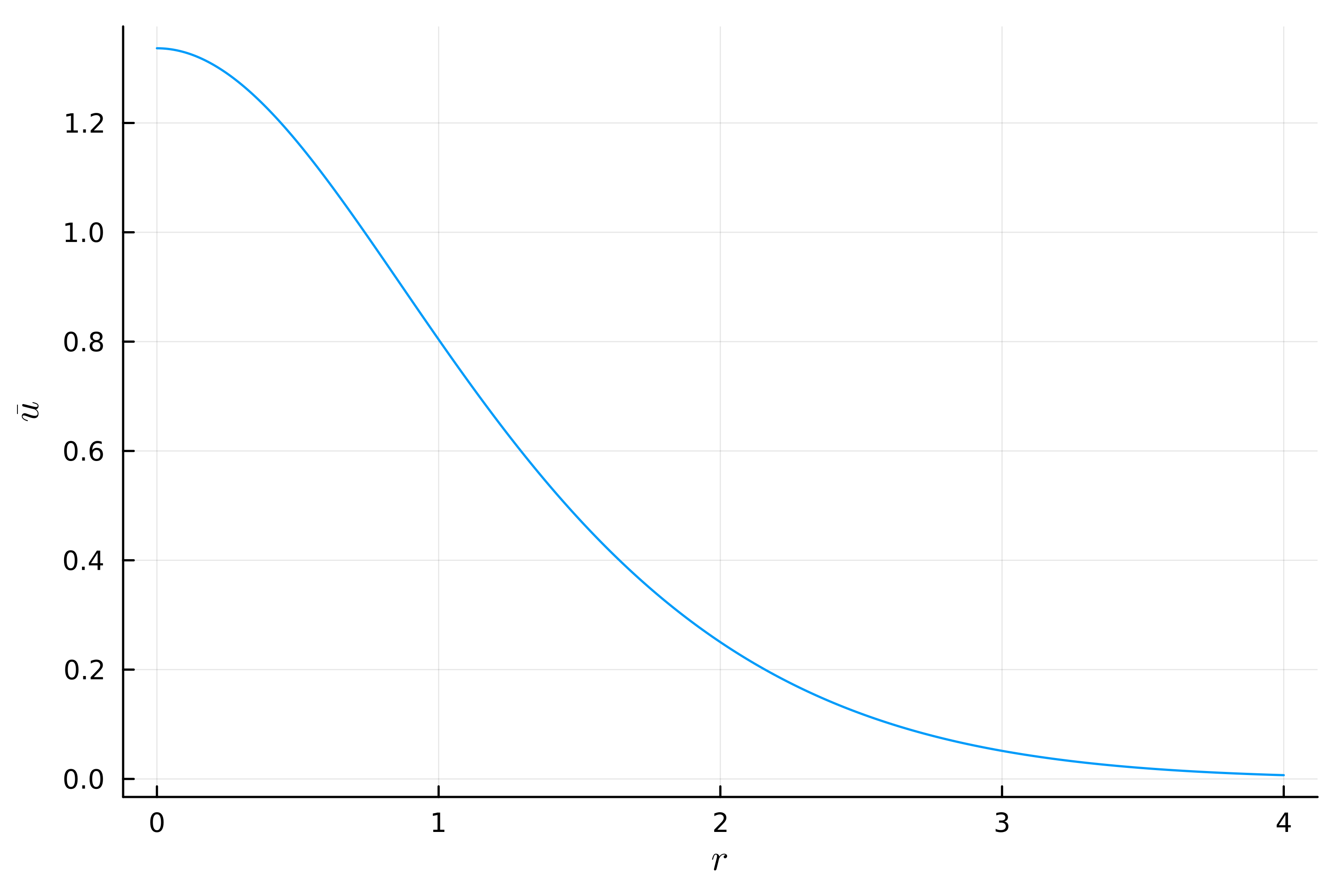} 
\caption{$d=2$, $p = 3$}
\label{fig:heat-rad-sol-2}
\end{subfigure}
\begin{subfigure}{0.5\textwidth}
\includegraphics[width=0.9\linewidth]{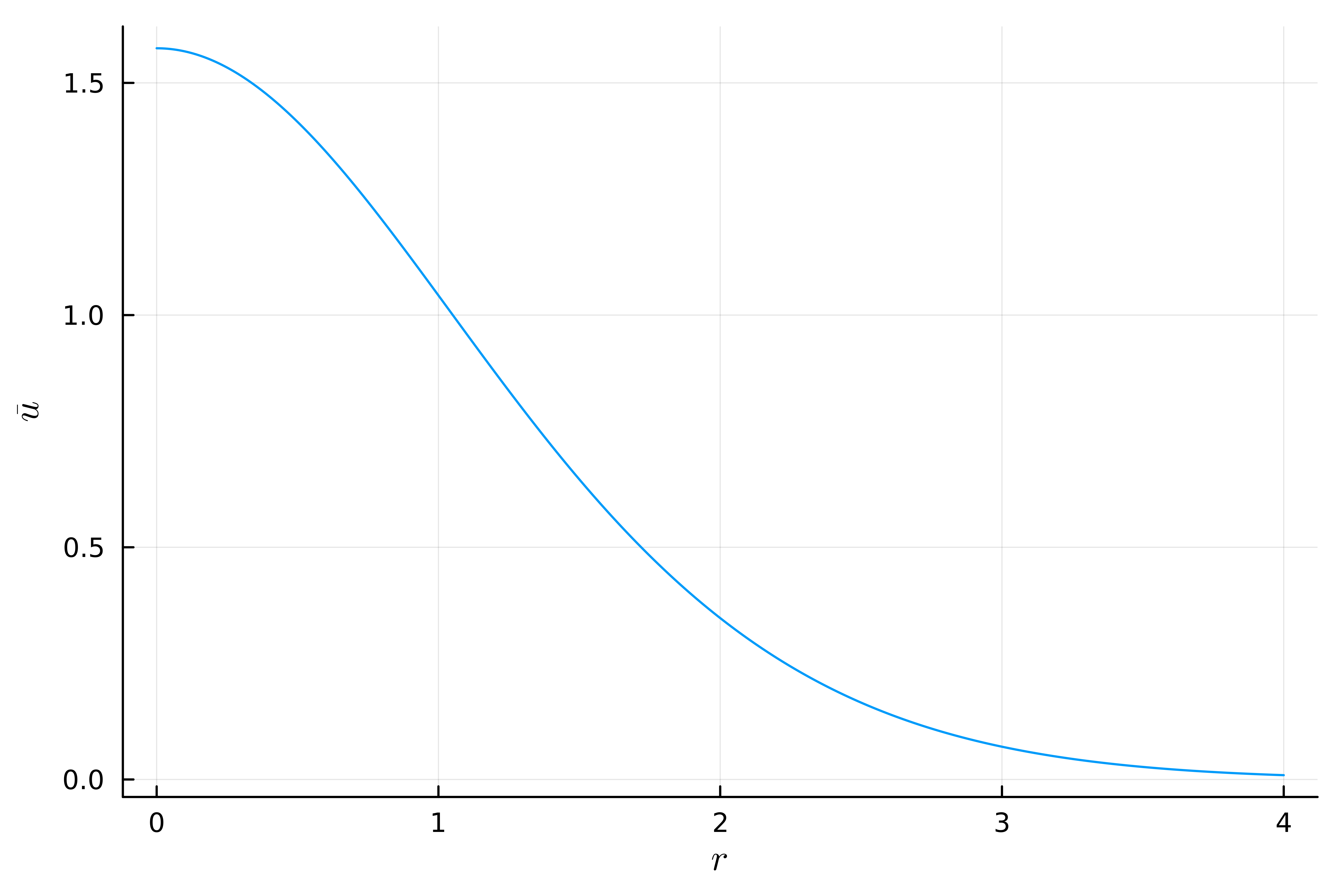}
\caption{$d=3$, $p=2$}
\label{fig:heat-rad-sol-3}
\end{subfigure}

\caption{Representation of the numerical approximate solutions $\bar{u}$ to Eq.~\eqref{eqn:rad-form-plus} which are validated in Theorem~\ref{thm:rad-int-sol}. Note that the spatial domain of course has to be truncated for the plot, but these approximate solutions are actually defined and computed on $[0,\infty)$ thanks to our choice of basis (see Corollary~\ref{cor:radial_basis}).\label{fig:heat-rad-sol}}
\end{figure}

\begin{table}[H]
\centering
\begin{tabular}{|c|c|c|}
    \hline
    & $\varepsilon = -1$, $d = 2$, $ p = 3$ & $\varepsilon = -1$, $d = 3$, $ p = 2$\\
    \hline
    $Y$ & $6.830397272197097\times10^{-13}$ & $1.4114627225937143\times 10^{-21}$\\
    \hline
    $Z_1$ & $0.009916643027658332$& $0.008621565176356588$ \\
    \hline
    $Z_2$ & $170.320283777825$& $17.04079355792355$ \\
    \hline
    $Z_3$ & $719.2983917297151$& 0 \\
    \hline
    $\underline{\delta}$ & $6.898810312173614\times 10^{-13}$ & $1.4237375688374744\times 10^{-21}$\\
    \hline
    $\bar{\delta}$ &$0.0057434127215418$ & $0.05817677630175133$\\
    \hline
\end{tabular}
\captionsetup{justification=centering}
\caption{Results for the a posteriori validation with the bounds given in Section~\ref{sec:bounds}\\ and $\bar{u}$ represented in Figure~\ref{fig:heat-rad-sol}.\label{tab:heat-rad-sol} 
}
\end{table}

\begin{thm}\label{thm:rad-int-sol}
    For $(d, p) = (2, 3)$ (resp. $(d, p) = (3,2)$), let $\bar{u}:\R_+\to\R$ be the radial function whose restriction to the interval is represented in Figure~\ref{fig:heat-rad-sol-2} (resp. Figure~\ref{fig:heat-rad-sol-3}) and whose precise description in terms of ``Fourier'' coefficients in the basis $\{\hat\psi_m\}_{m=0}^n$ is available at~\cite{Chu2024CodeOn}, then there exists a solution $u^{\star} \in \mathcal{H}$ to Eq.~\eqref{eqn:rad-form-plus} such that
    $$\|u^{\star} - \bar{u}\|_{H^2(\mu)}\leq \underline{\delta},$$
    where $\underline{\delta}$ is as in Table~\ref{tab:heat-rad-sol}. It is the only solution to Eq.~\eqref{eqn:rad-form-plus} in $\mathcal{H}$ such that $\|u^{\star} - \bar{u}\|_{H^2(\mu)}\leq \bar{\delta}$.
\end{thm}

\begin{proof}
    We consider $F :\mathcal{H}\to\mathcal{H}$ defined in~\eqref{eqn:rad-heat-F}. In the file \texttt{Heat/d2\_p3/proof.ipynb} (resp. \newline \texttt{Heat/d3\_p2/proof.ipynb}) at \cite{Chu2024CodeOn}, we construct $A_n$ (and thus indirectly $A$ and $T$) as outlined in Section~\ref{sec:fixedpoint}. Following the estimates of Section~\ref{sec:bounds}, we obtain the upper bounds in Table~\ref{tab:heat-rad-sol} for $Y$ and the $Z_k$'s to conclude with Corollary~\ref{cor:fixed-point}.
\end{proof}

\subsection{Proving positivity a posteriori}\label{sec:positivity}

Similarly to the previous literature on the semilinear heat equation, we also wish to verify the (strict) positivity of our radial solutions in this work. Note that even if $\bar{u}$ were proven to be strictly positive, for any $\delta>0$, $B(\bar{u},\delta)$ necessarily contains functions that are eventually negative at $+\infty$, e.g.~$\bar{u} - (\delta/2)\hat{\psi}_m/\|\hat{\psi}_m\|_{H^2(\mu)}$ for large enough $m$. Thus, the positivity of $u^{\star}$ cannot be deduced directly. Since $\bar{u}$ decays faster than our $L^{\infty}$-control, a direct comparison of $\bar{u}$ and $u^{\star}$ is only sufficient in a neighbourhood of zero. We thus resort to a maximum principle argument to show the (strict) positivity of $u^{\star}$ away from zero with the lemma below.

\begin{lem}\label{lem:max-princ}
    Let $b:\R^d\to\R^d$, $c:\R^d\to\R$ two $\mathcal{C}^0$ functions and $\varphi:\R^d\to\R$ a $\mathcal{C}^2$ function such that
    \begin{align}\label{eqn:phi-eq}
        -\Delta \varphi + b\cdot\nabla\varphi + c\varphi = 0.
    \end{align}
    Assume there exists $r_0>0$ such that $c(x)>0$ for all $\vert x\vert \geq r_0$. Assume further that $\varphi(x)>0$ for all $\vert x\vert \leq r_0$, and that $\varphi(x)\underset{\vert x\vert \to \infty}{\longrightarrow}0$. Then, $\varphi(x)>0$ for all $x\in\R^d$.
\end{lem}
This statement is essentially due to the maximum principle, which can be used even though we are on an unbounded domain because we already control the behaviour of the solution at infinity.
\begin{proof}
    We argue by contradiction and assume that $\varphi$ is nonpositive somewhere. Then, since $\varphi(x)\underset{\vert x\vert \to \infty}{\longrightarrow}0$, there exists $x_0\in\R^d$ such that $\varphi(x_0) = \min_{\R^d} \varphi \leq 0$, and we must have $\vert x_0 \vert > r_0$. Let $\mathcal{U}$ be a connected open subset of $\R^d$ containing $x_0$, such that $\vert x\vert \geq r_0$ for all $x\in \mathcal{U}$, and such that there exists $x$ in $\partial \mathcal{U}$ with $\vert x\vert = r_0$. Since $\varphi$ reaches its (nonpositive) minimum over $\bar{\mathcal{U}}$ at the interior point $x_0$, while $c$ is positive in $\mathcal{U}$, the strong maximum principle yields that $\varphi$ must be constant on $\bar{\mathcal{U}}$. In particular, $\varphi$ is nonpositive on $\partial \mathcal{U}$, which contradicts the fact that the $\varphi(x)$ is positive for all $\vert x \vert \leq r_0$.
\end{proof}

Note that with $\varepsilon=-1$ in~\eqref{eqn:rad-form}, we are in the ``bad-sign'' case of the semilinear heat equation, so that the strong maximum principle and the above lemma cannot be applied directly to equation~\eqref{eqn:carteq}. However, $\varphi = e^{|x|^2/8}u^{\star}$ satisfies an equation of the form~\eqref{eqn:phi-eq}, therefore its (strict) positivity can be verified via Lemma~\ref{lem:max-princ}, which in turn is equivalent to the (strict) positivity of $u^{\star}$. For clarity, let us denote below a radial solution $u^{\star}$ as a function of $r = |x|$.

\begin{cor}\label{cor:positivity}
    Let $u^{\star} \in H^1(\mu)$ be a solution of Eq.~\eqref{eqn:rad-form} and $r_0>0$ be such that
    \begin{equation}\label{eqn:c-pos}
    \frac{d}{4}+\frac{r^2}{16}-\frac{1}{p-1}+\varepsilon u^{\star}(r)^{p-1}>0\qquad \mbox{for all $r\geq r_0$}.
    \end{equation}
    If in addition, $u^{\star}(r)>0$ for all $r\leq r_0$, then $u^{\star}(r)>0$ for all $r\in \mathbb{R}_+$.
\end{cor}

\begin{proof}
    Define
    \begin{align*}
        \varphi :\R_+&\longrightarrow \R\\
        r&\longmapsto e^{r^2/8}u^{\star}(r).
    \end{align*}
    If $u^{\star}\in H^1(\mu)$ is a solution to Eq.~\eqref{eqn:rad-form}, then by~\cite[Theorem 3.12, Proposition 3.15]{Escobedo1987VariationalEquation} it is also a classical solution in the sense that $u^{\star}\in \mathcal{C}^2(\R^d)$. Since $u^{\star}(|\cdot|)\in \mathcal{C}^2(\R^d)$, we observe that $u^{\star} \in H^1(\R_+,e^{r^2/4})$ and thus from the proof of Lemma~\ref{lem:CartEmb} in the case $d=1$, we have that $\varphi\in H^1(\R_{+})$. Therefore, from~\cite[Corollary VIII.8]{Brezis2011FunctionalEquations}, $\lim_{r\to +\infty}\varphi(r) = 0$. Now, letting $b=0$ and
    $$c(x) = \frac{d}{4}+\frac{|x|^2}{16}-\frac{1}{p-1}+\varepsilon u^{\star}(|x|)^{p-1},$$
    we have that $\varphi(|\cdot|)$ solves the equation
    $$-\Delta \varphi + c \varphi = 0.$$
    Since $c, \varphi(|\cdot|) \in \mathcal{C}^2(\R^d)$ fulfil the assumptions of Lemma~\ref{lem:max-princ}, under condition~\eqref{eqn:c-pos}, we obtain that $\varphi(|x|)>0$ for all $x\in\R^d$ and thus $u^{\star}(r)>0$ for all $r \in \mathbb{R}_+$.
\end{proof}

In the case $d\in\{1, 2, 3\}$, it is straightforward to find some $r_0$ such that condition~\eqref{eqn:c-pos} holds by using the $L^{\infty}$-bound obtained from Lemma~\ref{lem:infty-bound} and using the estimate~\eqref{eqn:tail-estimate}. Similarly, checking $u^{\star}$ is positive on $[0, r_0]$ can be achieved efficiently by using Lemma~\ref{lem:infty-bound} and by using interval arithmetic to evaluate $P_m \bar{u}$ on $[0, r_0]$ for some small $m$ and using the estimate~\eqref{eqn:tail-estimate} to bound $|(I-P_m)\bar{u}(r)|$. Let us remark that alternatively (in particular for $d\geq 4$), $L^{\infty}$-estimates on $u^{\star}$ could be obtained by bootstrap arguments along the line of the proof of~\cite[Theorem 3.12]{Escobedo1987VariationalEquation}. By iteratively applying such arguments, one may also be able to obtain $\mathcal{C}^k$-bounds on $u^{\star}-\bar{u}$.

\begin{cor}
    Let $u^{\star}\in H^2(\mu)$ be as given by Theorem~\ref{thm:rad-int-sol}, $u^{\star}$ is strictly positive on $\R^d$.
\end{cor}

\begin{rmk}
    In the event that the positivity of $u^{\star}$ is verified, in the case of $\varepsilon = -1$, since $u^{\star}\in C^2(\R_+)$, we immediately obtain that it is a classical solution (in the sense of~\cite{Peletier1986On0}) of
    \begin{equation}
    \partial_r^2 u +\left(\frac{d-1}{r}+\frac{r}{2}\right)\partial_r u + \frac{u}{p-1} + |u|^{p-1}u = 0,
\end{equation}
    and since $u^{\star}$ is rapidly decaying, the results of~\cite{Peletier1986On0} can be applied and yield $u^{\star}(r) = \mathcal{O}(r^{2/(p-1)-d}e^{-r^2/4})$. Furthermore,~\cite[Theorem 1]{Yanagida1996UniquenessEquation} and~\cite{Naito2000RadialEquations} ensure that $u^{\star}$ is the unique rapidly decaying and strictly positive solution of Eq.~\eqref{eqn:rad-form} and of Eq.~\eqref{eqn:carteq} respectively. The analogous result for the case $\varepsilon = +1$ is given by~\cite{Brezis1986AAbsorption}.
\end{rmk}

\subsection{An example with a fractional exponent}
\label{sec:heat-frac}

In the case $\varepsilon = + 1$, rapidly decaying self-similar solutions only exist for $p<1+2/d$~\cite{Brezis1986AAbsorption, Escobedo1987VariationalEquation}, and therefore integer exponents are excluded as soon as $d>1$. We now explain how our method can be generalised to handle cases when the exponent $p$ is a rational number. We focus on enclosing the positive self-similar profile of the nonlinear heat equation with $\varepsilon = + 1$, $d= 2$ and $p = 5/3$. We thus aim to solve
\begin{equation}
    \mathcal{L}u -\frac{3}{2}u + u^{5/3} = 0,\label{eqn:frac-eq}
\end{equation}
where we choose the real branch $x \in \R\mapsto x^{5/3} := (\sqrt[3]{x})^2x$.
We first look for an approximate solution $\bar{u} = \bar{U}^3$ such that
$$\bar{U}(r) = q(r) e^{-r^2/12},$$
where $q$ is an even polynomial. The fractional exponent can then be \emph{rigorously} handled via a quadrature as in Lemma~\ref{lem:quad}. For instance, here, we first start by finding an approximate solution $\tilde{u}$ via a pseudospectral method~\cite{Funaro1991ApproximationFunctions}, and compute for $n\in 3\N$
$$q(r) = \sum_{m = 0}^{n/3}a_m L_m(r^2/12) \qquad \text{where } a_m \approx \frac{1}{6}\int_0^{\infty}L_m(r^2/12)\sqrt[3]{\tilde{u}(r)}r\d r,$$
so that $\bar{U} \approx \sqrt[3]{\tilde{u}}$. Finally, we define $\bar{u} := \bar{U}^3$ to apply Theorem~\ref{thm:fixed-point}. From here, $\bar{u}$ can be identified as a vector of $n+1$ intervals as in~\eqref{eqn:ubar-id} and computations can be carried out in the same manner as in the other examples. Note that here, we cannot expect $\bar{u}$ to achieve the minimal error $\min_{u\in P_n(\mathcal{H})}\|F(u)\|_{H^2(\mu)}$. Hence, we need to take more basis functions to achieve a comparable error as $\tilde{u}$, actually around $s$ times more where the degree of the nonlinearity is $p = q/s$.

\begin{rmk}
    It is clear that this approach can work with any rational exponent with an odd denominator. In the case of an exponent with an even denominator, this approach can only work if $\bar{U}$ can be verified to be nonnegative. Beyond the fact that this verification most likely requires a lengthy computation, finding such nonnegative $\bar{U}$ in the first place probably needs some fine-tuning. A more sophisticated approach would be to bypass this issue by reformulating the problem as a differential-algebraic equation as in~\cite{Breden2022Computer-assistedProblems}, but this would probably require additional bounds with an adapted functional analysis.
\end{rmk}
Now, analysing the function $x\mapsto x^{2/3}:= (\sqrt[3]{x})^2$, one finds that for all $x, y \in \R$,
\begin{equation*}\label{eqn:5-3-bound}
    \left|\diff{y^{5/3}}{y}-\diff{x^{5/3}}{x}\right|= \frac{5}{3}\left|y^{2/3}-x^{2/3}\right|\leq \frac{5}{3}|x-y|^{2/3}.
\end{equation*}
Thus, writing $\|DT(u)\|_{H^2, H^2} \leq \|DT(\bar{u})\|_{H^2, H^2} +\|DT(u) - DT(\bar{u})\|_{H^2, H^2}$ and estimating that for $h\in H^2(\mu)$
\begin{align*}
    \|(DT(u) - DT(\bar{u})) h\|_{H^2} &\leq \frac{5}{3}\|A\mathcal{L}^{-1}(u^{2/3}-\bar{u}^{2/3})h\|_{H^2}\\
    &\leq \frac{5}{3}\|\mathcal{L}A\mathcal{L}^{-1}\|_{L^2,L^2}\|(u^{2/3}-\bar{u}^{2/3})h\|_{L^2}\\
    &\leq \frac{5}{3}\|\mathcal{L}A\mathcal{L}^{-1}\|_{L^2,L^2}\|(u^{2/3}-\bar{u}^{2/3})\|_{\infty}\|h\|_{L^2}\\
    &\leq \frac{5}{3}\|\mathcal{L}A\mathcal{L}^{-1}\|_{L^2,L^2}\|(u-\bar{u})\|^{2/3}_{\infty}\|h\|_{L^2}\qquad \mbox{by~\eqref{eqn:5-3-bound}}\\
    &\leq \frac{5}{3}C(2)^{2/3}\|\mathcal{L}A\mathcal{L}^{-1}\|_{L^2,L^2}\|u-\bar{u}\|^{2/3}_{H^2}\|h\|_{H^2}\qquad\mbox{by Lemma~\ref{lem:infty-bound},}
\end{align*}
we may choose
$$\mathcal{Z}(\delta) := Z_1 + Z_{5/3}\delta^{2/3}$$
satisfying~\eqref{eqn:Zcond}, where $Z_1$ is as in Section~\ref{sec:bounds} and
$$Z_{5/3} = \frac{5}{3}C(2)^{2/3}\|\mathcal{L}A\mathcal{L}^{-1}\|_{L^2,L^2}.$$
For this example, we choose $n = 1200$ and find $\bar{u}$ depicted in Figure~\ref{fig:frac-sol}.

\begin{proof}[Proof of Theorem~\ref{thm:intro_frac}]
    We consider $F :\mathcal{H}\to\mathcal{H}$ defined in~\eqref{eqn:rad-heat-F}. In the file \texttt{Heat/fractional/proof.ipynb} at \cite{Chu2024CodeOn}, we load a precomputed $\bar{U}$ and compute $\bar{u}=\bar{U}^3$, then we construct $A_n$ (and thus indirectly $A$ and $T$) as outlined in Section~\ref{sec:fixedpoint}, and we find the following sufficient bounds:
    \begin{align*}
        Y &= 1.5917976375189734\times 10^{-23},\\
        Z_1 &= 0.002929002447260958,\\
        Z_{5/3} &=19.560776011689786,\\
        \underline{\delta} &=1.5964737129311592\times 10^{-23}\\
        \bar{\delta} &= 0.011508269024627918.
    \end{align*}
    We conclude with Theorem~\ref{thm:fixed-point}. Finally, positivity is deduced from Corollary~\ref{cor:positivity}.
\end{proof}

\subsection{An example of a non-radial solution}\label{sec:heat-asym}
Let us now turn to the problem of finding (real-valued) non-radial solutions to problem~\eqref{eqn:carteq}. {We limit ourselves to the case $\varepsilon = -1, d=2, p = 3$ here (though, such construction can be generalised), so that the problem remains polynomial while allowing for solutions that are not necessarily positive.} In polar coordinates, the equation for $u(r, \vartheta) = e^{-r^2/8}\varphi(r\cos\vartheta, r\sin\vartheta)$ reads
\begin{equation}
    \mathcal{L}u = \frac{1}{2} u + u^3,\label{eqn:asymmetric-heat}
\end{equation}
where
$$\mathcal{L}u= -\frac{r}{2}\pdiff{u}{r} - \Delta = -\frac{r}{2}\pdiff{u}{r} -\frac{1}{r}\pdiff{}{r}\left(r\pdiff{u}{r}\right) - \frac{1}{r^2}\pdiff{^2u}{\vartheta^2}.$$
Observe that since the equation itself is radially symmetric, if $u : \R_+ \times \mathbb{T} \to\R$ is a solution to Eq.~\eqref{eqn:asymmetric-heat}, then so is $u(\,\cdot\,,\, \cdot +\vartheta_0)$ for any phase $\vartheta_0\in \mathbb{T}$. Thus, our approach which relies on a contraction principle, and hence requires local uniqueness, cannot possibly work directly on $H^2(\mu)$. By treating the problem in polar coordinates, instead of using cartesian coordinates with the Hermite basis~\eqref{eqn:hermite-basis}, we get a natural way to handle the translation invariance in the angular variable.

We thus restrict the problem to an appropriate subspace, namely
$$\mathcal{H} = \overline{\mathrm{Span}\left\{\psi_{k, m}\,\mid\, k\in \N, m \in \N\right\}}^{H^2} \subset H^2(\mu),$$
where $\psi_{k, m}$ are as in Definition~\ref{def:spherical-basis}, and, as before, we denote by $\hat{\psi}_{k,m}$ the normalisation of $\psi_{k,m}$ with respect to the $L^2(\mu)$-inner product $\langle\cdot, \cdot\rangle$. Since we restricted ourselves to nonnegative $k$'s, $\mathcal{H}$ is nothing but the space of $H^2(\mu)$-functions which are even with respect to $\vartheta$. 
Solving for Eq.~\eqref{eqn:asymmetric-heat} in this space addresses the above symmetry ``issue'' and it is straightforward to check that
\begin{align}\label{eqn:F-heat-asym}
    F : \mathcal{H} &\longrightarrow \mathcal{H} \nonumber\\
    u &\longmapsto u - \mathcal{L}^{-1}\left(\frac{u}{2}+u^3\right)
\end{align}
is well-defined and compact. In this respect, our strategy goes along the same line as the one employed in~\cite{Arioli2019Non-radialDisk} to find non-radial solutions of analogous semilinear equations on the disc. In the most general setting, the above-defined space $\mathcal{H}$ shall be used, but one can simplify the proof further if the solution has a zero projection onto certain frequencies. In our example, we will perform our computer-assisted proof in the subspace

$$\tilde{\mathcal{H}} = \overline{\mathrm{Span}\left\{\psi_{2k+1, m}\,\mid\, k\in \N, m \in \N\right\}}^{H^2} \subset \mathcal{H}\subset H^2(\mu).$$

Working on such subspaces can also automatically discard other non-zero solutions which are not of interest such as radial ones. Here, we have that
$$\mathcal{L}\psi_{2k+1,m} = \lambda_{k+m}\psi_{2k+1, m}, \qquad \lambda_{j} := \frac{3}{2}+j,$$
and the $\psi_{2k+1, m}$ for $k,m\in\mathbb{N}$ form a Hilbert basis of $\tilde{\mathcal{H}}$. Note that the definition of $\lambda_j$ in that context is slightly different from the one introduced in Notation~\ref{not:2}. In particular, this change reflects the fact that the first eigenvalue of $\mathcal{L}$ restricted to $\tilde{\mathcal{H}}$ is strictly larger than $d/2=1$, which slightly improves the Poincaré inequalities~\eqref{eqn:Poincineq} and~\eqref{eqn:Poincineq2} on $\tilde{\mathcal{H}}$.

The estimates for the radii polynomials coefficients are essentially the same as in Section~\ref{sec:bounds}, with $P_n$ now denoting the projection onto $\mathrm{Span}\left\{\psi_{2k+1, m}\,\mid\, k\in \N, m \in \N \text{ with } k+m\leq n\right\}$, and
$$w_{k,m} := \|\OP (\bar{u}^{2}\hat{\psi}_{2k+1,m})\|_{L^2(\mu)} =  \left(\|\bar{u}^{2} \hat{\psi}_{2k+1,m}\|^2_{L^2(\mu)}-\|P_n(\bar{u}^{2}\hat{\psi}_{2k+1,m})\|^2_{L^2(\mu)}\right)^{1/2}.$$
The computation of integrals described in Section~\ref{sec:quadrature} is adapted as follows:
\begin{itemize}
    \item the integral is first reformulated as a sum of integrals of product of functions, separable in $r$ and $\vartheta$ components so as to apply Fubini's Theorem,
    \item the integrals with respect to $\vartheta$ are treated via the multiplication formula for cosines and orthogonality,
    \item the integrals with respect to $r$ are of the form~\eqref{eqn:quad} after the change of variable $z = r^2/4$.
\end{itemize}
As for the $L^{\infty}$-bounds on $\hat{\psi}_{2k+1, m}$, we have that $\vert\hat{\psi}_{2k+1, m}(r,\vartheta)\vert \leq \vert\hat{\psi}_{2k+1, m}(r,0)\vert$, and we compute the local extrema of $r\mapsto \partial_r\hat{\psi}_{2k+1, m}(r,0)$ by enclosing its zeros (which must all be real), this time using \texttt{PolynomialRoots.jl}, the Julia implementation of~\cite{Skowron2012GeneralMicrolenses}.

\begin{rmk}
    In practice, while we might do the computer-assisted proof with many basis functions, we set some of the entries of $\bar{u}$ to zero to eliminate unnecessary computational costs so that the validation can be performed within a manageable amount of time on a personal computer.
\end{rmk}

We choose $n = 150$ (i.e.~$(n+1)(n+2)/2=11 476$ basis functions) and obtain the approximate solution $\bar{u}$ depicted in Figure~\ref{fig:asymmetric-heat}. 

\begin{figure}[h!]
    \centering
    \includegraphics[width=0.6\textwidth]{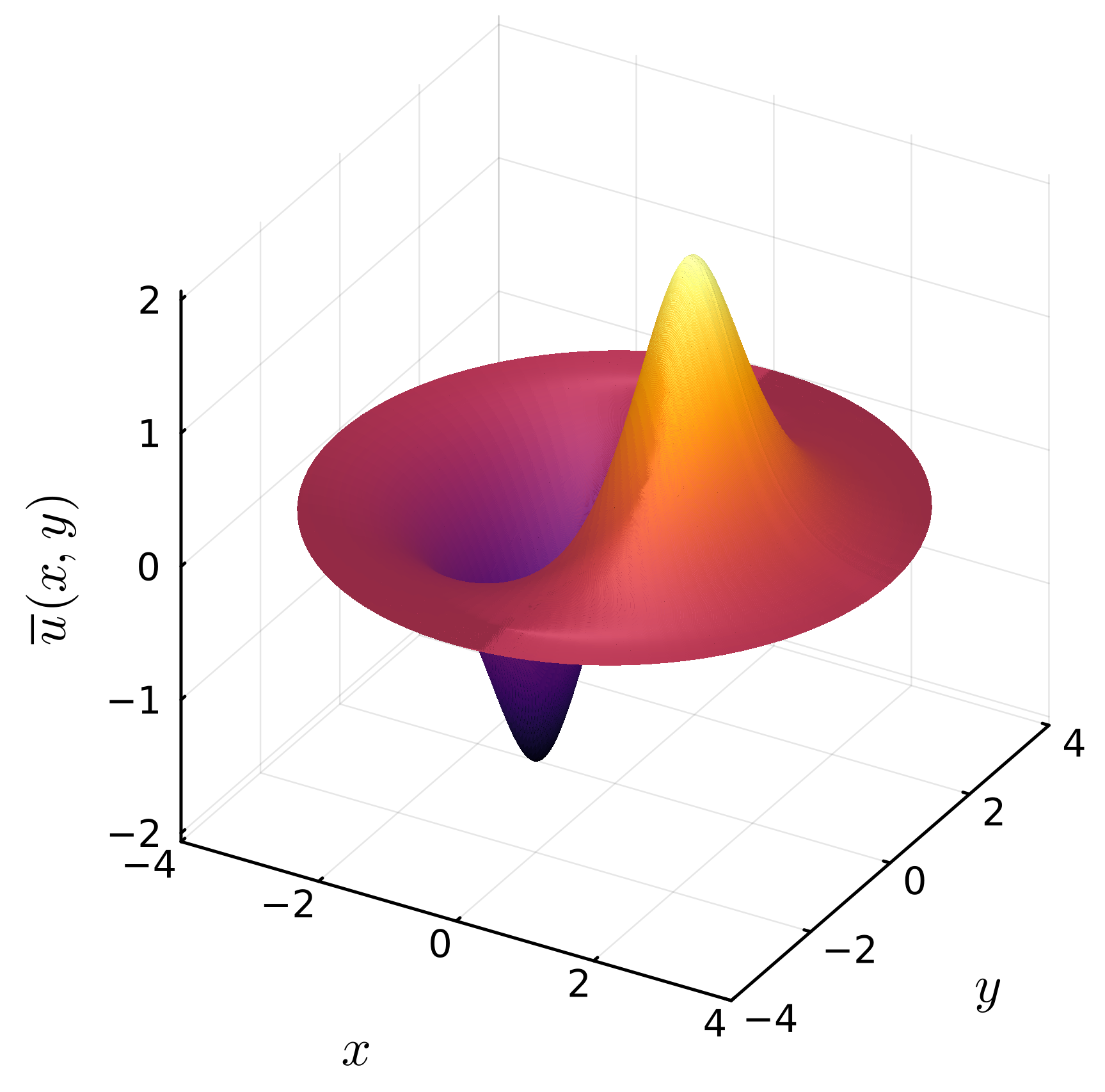}
    \caption{A non-radial numerical solution $\bar{u}$ to Eq.~\eqref{eqn:asymmetric-heat}.\label{fig:asymmetric-heat}}
\end{figure}

\begin{thm}
\label{thm:asym-heat}
    Let $\bar{u}:\mathbb{R}^2\to\mathbb{R}$ be the function whose restriction to the disk $\{(x, y)\mid x^2+y^2 \leq 4^2\}$ is represented on Fig~\ref{fig:asymmetric-heat} and whose description in terms of ``Fourier'' coefficients in $\{\hat{\psi}_{2k+1, m}\}_{k\geq 0, k+m\leq n}$ is available at~\cite{Chu2024CodeOn}. There exists a solution $u^{\star}\in \tilde{\mathcal{H}}$ to Eq.~\eqref{eqn:asymmetric-heat} such that

    $$\|u^{\star} - \bar{u}\|_{H^2(\mu)} \leq 1.6\times 10^{-3}.$$
\end{thm}

\begin{proof}
    We consider $F :\tilde{\mathcal{H}}\to\tilde{\mathcal{H}}$ defined in~\eqref{eqn:F-heat-asym}. In the file \texttt{Heat/asymmetric/proof.ipynb} at \cite{Chu2024CodeOn}, we construct $A_n$ (and thus indirectly $A$ and $T$) as outlined in Section~\ref{sec:fixedpoint}, we find the following sufficient bounds:
    \begin{align*}
        Y &= 0.0011482317939412424,\\
        Z_1 &= 0.09887537542580953,\\
        Z_2 &= 236.2041502678645,\\
        Z_3 &= 456.22972624236917,\\
        \bar{\delta} &= 0.003801071181766125,
    \end{align*}
    and we conclude with Corollary~\ref{cor:fixed-point}.
\end{proof}
The relatively large error bound obtained here in Theorem~\ref{thm:asym-heat} (for instance compared to that in Theorem~\ref{thm:intro_asym}), is related to the fact that the decay of the coefficients of the solution is worst in this case, see Figure~\ref{fig:decay}. Therefore, increasing $n$ further would only very slowly decrease the residual error $\|F(\bar{u})\|_{H^2(\mu)}$, while significantly increasing computational cost.

\begin{figure}[h!]
    \centering
\includegraphics[width=0.5\textwidth]{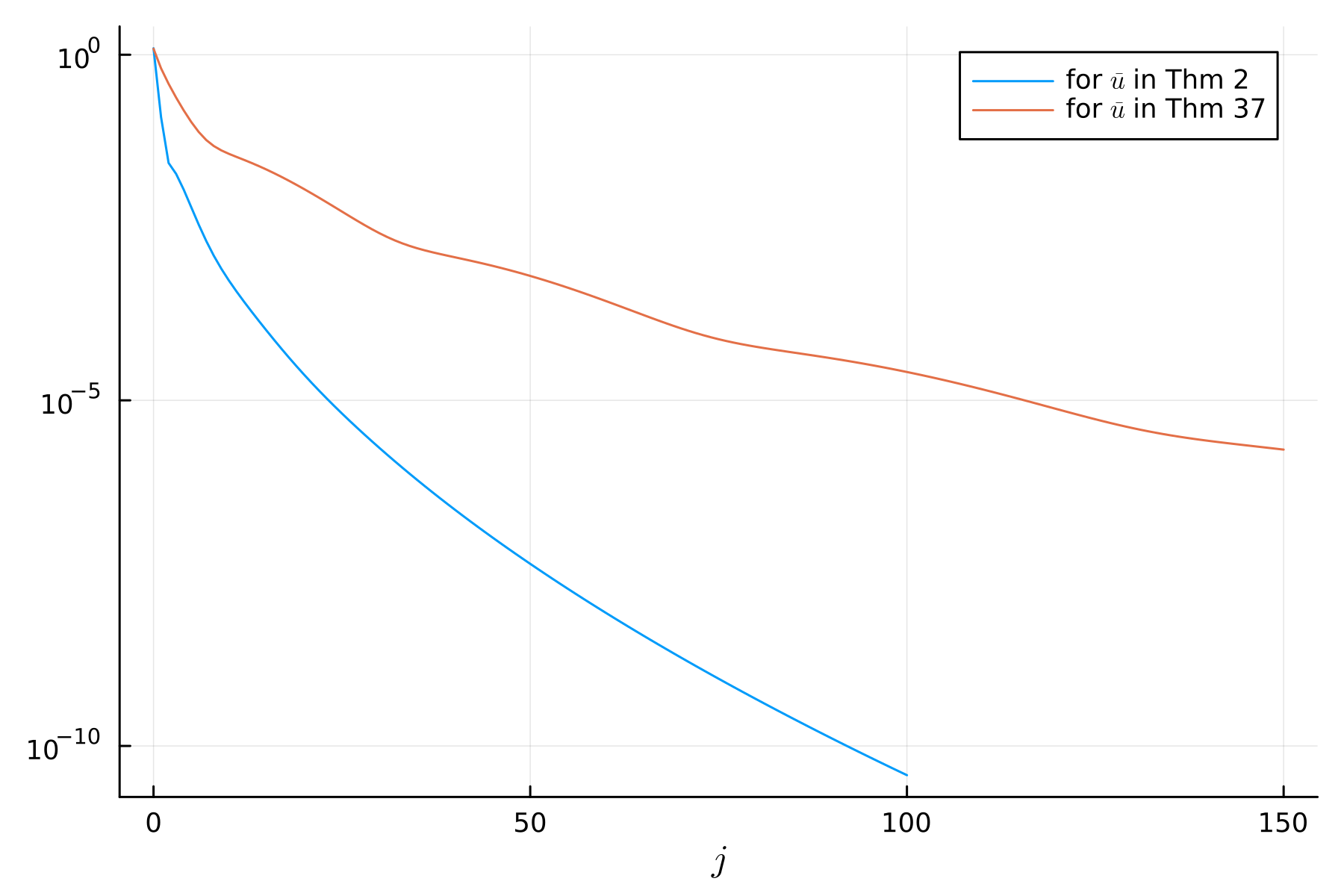}
    \caption{For each $j$, we display here the $L^2(\mu)$-norm of the coefficients of $\bu$ associated with eigenfunctions having eigenvalue $\lambda_j$, that is, $\displaystyle\sqrt{\sum_{k+m=j}|\langle\bu, \hat{\psi}_{2k+1,m}\rangle|^2}$. }
    \label{fig:decay}
\end{figure}

\section{Self-similar solutions of the nonlinear Schr\"{o}dinger equation}
\label{sec:Schrodinger}

Let us now turn to the critical nonlinear Schr\"{o}dinger equation
\begin{equation}
    i\partial_t v + \Delta v +\varepsilon|v|^{4/d}v = 0, \label{eqref:nonlin-schrodinger2}
\end{equation}
for $\varepsilon = \pm 1$. It is known~\cite{Kavian1994Self-similarEquation} that this equation admits self-similar solutions of the form
\begin{equation}
    v(t, x) = (1+t^2)^{-d/4}\exp\left(\frac{it|x|^2}{4(1+t^2)}\right)e^{i\omega\arctan{t}}\varphi\left(\frac{\sqrt{2}x}{(1+t^2)^{1/2}}\right),\label{eqn:schrodinger-transformation}
\end{equation}
for any $\omega\in\R$, as soon as $\varphi : \R^d \longrightarrow \mathbb{C}$ solves
\begin{equation}\label{eqn:phi-schrodinger}
-\Delta \varphi +\frac{|x|^2}{16}\varphi +\frac{\omega}{2} \varphi = \frac{\varepsilon}{2}|\varphi|^{4/d}\varphi.
\end{equation}
Assuming $\varphi(x) = e^{|x|^2/8}u(x)$, we get the equation 
\begin{equation}\label{eqn:schrodinger-u-cart}
    \mathcal{L}u = \left(\frac{d}{4}-\frac{\omega}{2}\right) u +\frac{\varepsilon}{2}e^{(p-1)|x|^2/8}|u|^{4/d}u,
\end{equation}
where $p = 4/d+1$.

\begin{rmk}
    In~\cite{Kavian1994Self-similarEquation}, Eq.~\eqref{eqn:phi-schrodinger} is solved via variational methods for $\varphi$ in the Hilbert space
    $$\mathcal{X} = \left\{ \varphi \in H^1(\R^d) : |x|\varphi \in L^2(\R^d)\right\},$$
    but this is equivalent to solving Eq.~\eqref{eqn:schrodinger-u-cart} for $ u \in H^1(e^{|x|^2/4})$. Furthermore, an enclosure of a solution $u^{\star}$ to~\eqref{eqn:schrodinger-u-cart} around $\bar{u}$ with respect to the $H^2(\mu)$-norm will automatically yield an enclosure for a solution $\varphi^{\star} := e^{|x|^2/8} u^{\star}$ to~\eqref{eqn:phi-schrodinger} around $\bar{\varphi} := e^{|x|^2/8}\bar{u}$ with respect to the $H^2(\R^d)$-norm by inequality~\eqref{eqn:Hessian-ineq}.
\end{rmk}

In the radial case, we thus aim to find a non-trivial zero of
\begin{equation}\label{eqn:F-Schrodinger}
F(u) = u  - \mathcal{L}^{-1}\left(\left(\frac{d}{4}-\frac{\omega}{2}\right)u +\frac{\varepsilon}{2}e^{(p-1)r^2/8}u^p\right).
\end{equation}
We get similar estimates for the radii polynomials coefficients as in Section~\ref{sec:bounds} and take
\begin{align*}
    Y &= \left\{\|\mathcal{L}A_n(P_nF(\bar{u}))\|^2_{L^2(\mu)}+\frac{1}{4}\left(\|e^{(p-1)r^2/8}\bar{u}^p\|^2_{L^2(\mu)}-\|P_n(e^{(p-1)r^2/8}\bar{u}^p)\|^2_{L^2(\mu)}\right)\right\}^{1/2}\\
    Z^{11} &= \|\mathcal{L}(I_n - A_n P_nDF(\bar{u})P_n)\mathcal{L}^{-1}\|_{L^2(\mu),L^2(\mu)}\\
    Z^{21} &=\frac{p}{2}\|\mathcal{L}^{-1}w\|_{L^2(\mu)}\\
    Z^{12} &= \frac{p}{2\lambda_{n+1}}\left\|\,\left|\mathcal{L}A_n \mathcal{L}^{-1}\right|w\right\|_{L^2(\mu)}\\
    Z^{22} &= \frac{\|d/4-\omega/2+\varepsilon p(e^{r^2/8}\bar{u})^{p-1}/2\|_{\infty}}{\lambda_{n+1}}\\
    Z_k &= \frac{p!}{d(p-k)!}C(d)^{k-1}\|\mathcal{L}A\mathcal{L}^{-1}\|_{L^2(\mu),L^2(\mu)}\|e^{r^2/8}\bar{u}\|_{\infty}^{p-k}  \\
    \text{or }Z_k &= \frac{p!}{2(p-k)!}\left(\frac{2}{d}\right)^{k/2}c(d, k)\|\mathcal{L}A\mathcal{L}^{-1}\|_{L^2(\mu),L^2(\mu)}\|e^{r^2/8}\bar{u}\|_{\infty}^{p-k},
\end{align*}
where
$$w_m := \|\OP ((e^{r^2/8}\bar{u})^{p-1}\hat{\psi}_m)\|_{L^2(\mu)} =  \left(\|(e^{r^2/8}\bar{u})^{p-1} \hat{\psi}_m\|^2_{L^2(\mu)}-\|P_n((e^{r^2/8}\bar{u})^{p-1}\hat{\psi}_m)\|^2_{L^2(\mu)}\right)^{1/2}.$$
Similarly, the quadrature rule of Section~\ref{sec:quadrature} can be adapted. For instance, the matrix
    $$G_{ij} = \langle \hat{\psi}_i, e^{(p-1)r^2/8}\bar{u}^{p-1}\hat{\psi}_j\rangle$$
can be written as $G = \bar{V}^{T}\mathrm{Diag}(\bar{V} \bar{\mathbf{u}})^{p-1}\bar{V}$ for $\bar{V}$ as in Section~\ref{sec:quadrature}, with $m =p+1$ and $\beta = (p+1)/2$. One could also study the positivity of our radial solutions by applying Lemma~\ref{lem:max-princ} to Eq.~\eqref{eqn:phi-schrodinger}.

Furthermore, one can find non-radial solutions by adapting the setup and the above bounds analogously to Section~\ref{sec:heat-asym}.

\begin{figure}[h]
\begin{subfigure}{0.5\textwidth}
\includegraphics[width=0.9\linewidth]{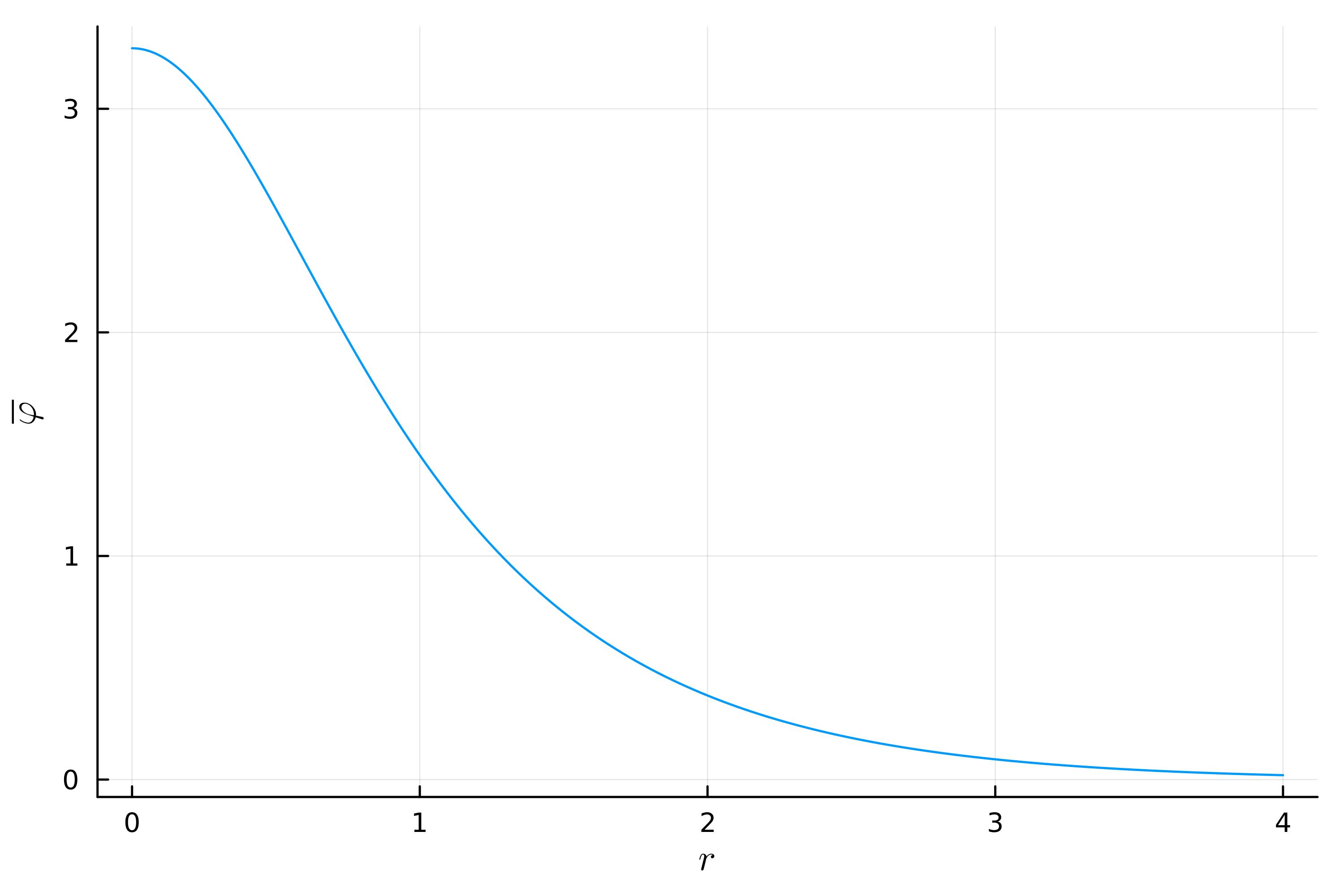} 
\caption{$\varepsilon = +1$, $d = 2$, $ \omega = 2$}
\label{fig:schrodinger-rad-sol-1}
\end{subfigure}
\begin{subfigure}{0.5\textwidth}
\includegraphics[width=0.9\linewidth]{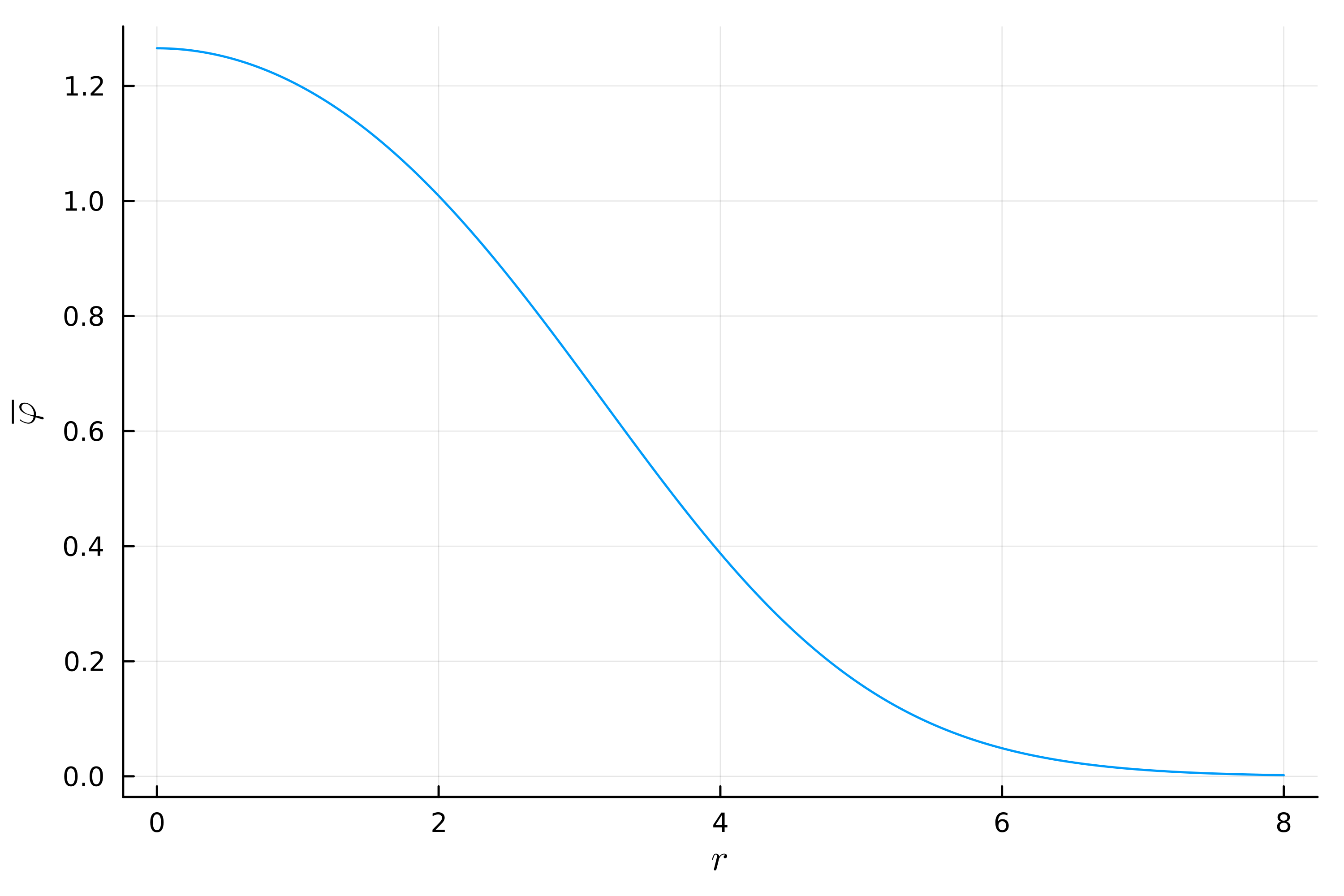} 
\caption{$\varepsilon = -1$, $d = 2$, $ \omega = -2$}
\label{fig:schrodinger-rad-sol-2}
\end{subfigure}

\caption{Plots of the numerical approximations $\bar{\varphi} = e^{|x|^2/8}\bar{u}$ to Eq.~\eqref{eqn:phi-schrodinger}.\label{fig:schrodinger-rad-sol}}
\end{figure}

\begin{table}[H]
\centering
\begin{tabular}{|c|c|c|c|}
    \hline
    & $\varepsilon = +1$, $d = 2$, $ \omega = 2$ & $\varepsilon = -1$, $d = 2$, $ \omega = -2$ & $\varepsilon = -1$, $d = 2$, $ \omega = -5/2$ \\
    \hline
    $Y$ & $1.1097624368731403\times 10^{-7}$ & $1.510794398394117\times 10^{-33}$ & $7.709839958677216\times 10^{-9}$\\
    \hline
    $Z_1$ & $0.015947448213879025$& $0.009444194874814784 $ & $0.04065143211428217$\\
    \hline
    $Z_2$ & $77.79632835336696$& $468.0575302222208 $ & $109.71390801884884$\\
    \hline
    $Z_3$ & $134.26940590853712$& $1264.711434525654 $ & $486.14503594089473$\\
    \hline
    $\underline{\delta}$ & $1.1277557297872954\times 10^{-7} $ & $1.5251987173285673\times 10^{-33}$ & $8.036540516245275\times10^{-9}$\\
    \hline
    $\bar{\delta}$ &$0.007787422383319831$ & $0.0021102955124448573$ & $0.008580958099091766$\\
    \hline
    $\eta$ & $1.8917745266839649\times 10^{-6}$ & $7.152376831616066\times 10^{-33}$ & $ 3.768713253012379\times 10^{-8}$\\
    \hline
    $\mathrm{dim}(P_n\mathcal{H})$ & $501$ & $501$ & $5151$\\
    \hline
    Figure &\ref{fig:schrodinger-rad-sol-1} &\ref{fig:schrodinger-rad-sol-2} & \ref{fig:asymmetric-schrodinger}\\
    \hline
\end{tabular}
\captionsetup{justification=centering}
\caption{Results for the a posteriori validation of the solutions to Eq.~\eqref{eqn:phi-schrodinger} in Theorem~\ref{thm:Schrodinger}.\label{table:schrodigner}}
\end{table}

\begin{thm}
\label{thm:Schrodinger}
    For $d=2$ and $(\varepsilon, \omega) = (+1, 2)$, (resp. $(\varepsilon, \omega) = (-1, -2)$, resp. $(\varepsilon, \omega) = (-1, -5/2)$), let $\bar{\varphi}$ be the function depicted on a compact set in Figure~\ref{fig:schrodinger-rad-sol-1} (resp.~\ref{fig:schrodinger-rad-sol-2}, resp.~\ref{fig:asymmetric-schrodinger}) and whose description in terms of ``Fourier'' coefficients is available at~\cite{Chu2024CodeOn}. There exists a solution $\varphi^{\star}\in H^2(\R^d)$ to Eq.~\eqref{eqn:phi-schrodinger} such that
    $$\|\varphi^{\star} - \bar{\varphi}\|_{H^2(\R^2)}\leq \eta,$$
    where $\eta$ is as in Table~\ref{table:schrodigner}.
\end{thm}
\begin{proof}
    We consider $F$ defined in~\eqref{eqn:F-Schrodinger}. In the file \texttt{Schrodinger/plus/proof.ipynb} (resp. \newline\texttt{Schrodinger/minus/proof.ipynb}, resp. \texttt{Schrodinger/asymmetric/proof.ipynb}) at \cite{Chu2024CodeOn}, we construct $A_n$ (and thus indirectly $A$ and $T$) as outlined in Section~\ref{sec:fixedpoint} and following the estimates above, we obtain the upper bounds in Table~\ref{table:schrodigner} for $Y$ and the $Z_k$'s to conclude with Corollary~\ref{cor:fixed-point}. Now, observe that from combining~\eqref{eqn:full-H2-ineq} with the Poincaré inequalities~\eqref{eqn:Poincineq} and~\eqref{eqn:Poincineq2}, for $d\geq 2$, for all $u\in H^2(\mu)$
    and $\varphi = e^{|x|^2/8}u$, we have that%
    $$\|\varphi\|_{H^2(\R^d)}^2 := \|D^2\varphi\|_{L^2(\R^d)}^2 +\|\nabla \varphi\|_{L^2(\R^d)}^2+\|\varphi\|_{L^2(\R^d)}^2\leq Z\left[1+\max\left(0,1-\frac{d}{8}\right)\frac{4}{d^2}\right]\| u\|_{H^2(\mu)}^2.$$
    And thus, for $d=2$, we may choose $\eta = (\sqrt{7Z}/2)\underline{\delta} $.
\end{proof}
Note that the third case of Theorem~\ref{thm:Schrodinger} corresponds to Theorem~\ref{thm:intro_asym}.

\section{An example with first order terms}
\label{sec:Burger}

In this final section, we consider a toy problem whose interest lies in the fact that it leads to an equation of the form~\eqref{eqn:general_elliptic_prblm} with a first-order term, i.e. a right-hand side of the form $f(x,u,\nabla u)$ with an actual dependence on $\nabla u$. We focus on a generalised viscous Burgers equation on $\R_+$ with Neumann boundary condition
\begin{equation}\label{eqn:burger2}
    \begin{cases}
        \partial_t v + v^2\partial_x v = \partial_{xx} v,\\
        \partial_x v(t, 0) = 0 \qquad\qquad\quad\mbox{for all $t>0.$}
    \end{cases}
\end{equation}
Then, under the similarity transformation
$$v(t, x) = t^{-1/4}u\left(\frac{x}{\sqrt{t}}\right),$$
$v$ solves Eq.~\eqref{eqn:burger2} for $t>0$ if $u$ solves
\begin{equation}\label{eqn:burger3}
    \mathcal{L}u-\frac{u}{4}+u^2\frac{\partial u}{\partial x} = 0, \qquad x\geq 0.
\end{equation}
Similar problems have been studied in~\cite{Aguirre1990Self-SimilarProblems, SrinivasaRao2003Self-similarDamping} and the references therein.
This equation can be approached with our method, as $u\in H^2(\mu)$ implies that $u\in \mathcal{C}^{1}_0(\R)$ and thus $u^2\partial_x u \in L^2(\mu)$. 
\begin{rmk}
    We deal with a one-dimensional example here for simplicity, but terms like $(u\cdot\nabla)u$ could in principle also be handled in dimension $d\in\{2,3\}$, as $u\in H^2(\mu)$ is then still enough to guarantee that $(u\cdot\nabla)u\in L^2(\mu)$ since $u\in L^{\infty}(\R^d)$.
\end{rmk}
In order to enforce the Neumann boundary condition, we symmetrise the problem and work on the Hilbert space
$$\mathcal{H} 
= \overline{\mathrm{Span}\{\hat{\psi}_m\}_{m=0}^{\infty}}^{H^2} = \overline{\mathrm{Span}\{e^{-x^2/4}\hat{H}_{2m}(x/2)\}_{m=0}^{\infty}}^{H^2},$$
i.e.~we make use of ``even half-Hermite series'' where $\hat{H}_{n} = H_{n}/\sqrt{2^{n}n!\sqrt{\pi}}$ denotes a normalised Hermite polynomial. Note that working on $\mathcal{H}$ is equivalent to working on $H^2(\R_+, e^{x^2/4})$ with the additional Neumann boundary condition (see also~\cite[Section 3]{Weissler1986RapidlyEquations}). Furthermore, for clarity, in what follows, all elements of $L^2(\mu)$ that we will consider should be treated as functions on $\R_+$ having an even extension on $\R$. In particular, for such functions $u$ and $v$, we have that
$$\langle u, v \rangle = \int_{-\infty}^{\infty}uv\mu(x) \d x = \frac{1}{2}\int_{-\infty}^{\infty}uve^{x^2/4} \d x = \int_{0}^{\infty} u v e^{x^2/4}\d x.$$
We consider
\begin{align}
    F:\mathcal{H}&\longrightarrow \mathcal{H} \nonumber\\
    u &\longmapsto u + \mathcal{L}^{-1}\left(-\frac{u}{4} + u^2\partial_x u\right),\label{eqn:F-burger}
\end{align}
and construct our fixed-point operator $T:\mathcal{H}\to \mathcal{H}$ as outlined in~\ref{sec:fixedpoint}. Note that $T$ is still polynomial with respect to $u$ in this case and most bounds for the validation can be derived similarly as in Section~\ref{sec:bounds}:
\begin{align*}
    Y &= \left\{\|\mathcal{L}A_n(P_nF(\bar{u}))\|^2_{L^2(\mu)}+\left(\|\bar{u}^2\partial_x\bar{u}\|^2_{L^2(\mu)}-\|P_n(\bar{u}^2\partial_x\bar{u})\|^2_{L^2}\right)\right\}^{1/2},\\
    Z^{11} &= \|\mathcal{L}(I_n - A_n P_nDF(\bar{u})P_n)\mathcal{L}^{-1}\|_{L^2(\mu),L^2(\mu)},\\
    Z^{21} &=\|\mathcal{L}^{-1}w\|_{L^2(\mu)},\\
    Z^{22} &= \frac{\|1/4+2\bar{u}\partial_x \bar{u}\|_{\infty}}{\lambda_{n+1}} +\frac{\|\bar{u}\|_{\infty}^2}{\sqrt{\lambda_{n+1}}},\\
     Z_2 &= 2^{15/4}\|\mathcal{L}A\mathcal{L}^{-1}\|_{L^2(\mu),L^2(\mu)}(\sqrt{2}\|\bar{u}\|_{\infty} +\|\partial_x \bar{u}\|_{\infty}),\\
     Z_3 &= 96\|\mathcal{L}A\mathcal{L}^{-1}\|_{L^2(\mu),L^2(\mu)},
\end{align*}
with
\begin{align*}
    w_m :&= 2\|\OP(\bar{u}\hat{\psi}_m\partial_x\bar{u})\|_{L^2}+\|\OP  (\bar{u}^2\partial_x\hat{\psi}_m)\|_{L^2}\\
    &=  2\left(\|\bar{u}\hat{\psi}_m\partial_x\bar{u}\|^2_{L^2}-\|P_n(\bar{u}\hat{\psi}_m\partial_x\bar{u})\|^2_{L^2}\right)^{1/2}+  \left(\|\bar{u}^2\partial_x\hat{\psi}_m\|^2_{L^2}-\|P_n(\bar{u}^2\partial_x\hat{\psi}_m)\|^2_{L^2}\right)^{1/2}.
\end{align*}
The only estimate which requires substantial adaptations is $Z^{12}$. For $\ih \in \iH$, we get
\begin{align}
\label{eq:Z12_Burger}
    \langle 2 \bar{u}\partial_x\bar{u} \ih + \bar{u}^2\partial_x \ih,  \hat{\psi}_m \rangle &= \langle \partial_x (\bar{u}^2\ih), \hat{\psi}_m\rangle \nonumber\\
    &=\int_0^{\infty}\partial_x(\bar{u}^2\ih)\hat{\psi}_m e^{x^2/4}\mathrm{d} x \nonumber\\
    &= \bar{u}(0)^2\ih(0)\hat{\psi}_m(0) - \int_0^{\infty}\bar{u}^2\ih\partial_x(\hat{\psi}_m e^{x^2/4})\mathrm{d} x \qquad \mbox{by parts.}
\end{align}
We deal with both terms appearing in~\eqref{eq:Z12_Burger} separately. First, writing $\ih$ as 
$$ \ih = \sum_{m = n+1}^{\infty}a_m \hat{\psi}_m,$$
we have that
\begin{align*}
    \ih(0)^2 &= \left(\sum_{m = n+1}^{\infty}a_m \hat{\psi}_m(0)\right)^2\\
    &\leq\left(\sum_{m = n+1}^{\infty}\lambda_m^2a_m^2\right)\left(\sum_{m = n+1}^{\infty}\frac{\hat{\psi}_m(0)^2}{\lambda_m^2}\right)\qquad \mbox{by Cauchy--Schwarz inequality}\\
    &=\frac{\|\ih\|_{H^2(\mu)}^2}{\Gamma(1/2)^2}\sum_{m=n+1}^{\infty}\frac{\Gamma(m+1/2)}{(m+1/2)^2m!} \qquad\mbox{by~\eqref{eqn:tail-estimate}}\\
    &= \frac{\|\ih\|_{H^2(\mu)}^2}{\pi}\left(\sum_{m=0}^{\infty}\frac{\Gamma(m+1/2)}{(m+1/2)^2m!}-\sum_{m=0}^{n}\frac{\Gamma(m+1/2)}{(m+1/2)^2m!}\right)\\
    &\leq \frac{\|\ih\|_{H^2(\mu)}^2}{\pi}\left(\pi^{3/2}\log{4}-\sum_{m=0}^{n}\frac{\Gamma(m+1/2)}{(m+1/2)^2m!}\right) \\
    &\leq \|\ih\|_{H^2(\mu)}^2S_n,
\end{align*}
where
$$ S_n = \frac{1}{\pi}\left(\pi^{3/2}\log{4}-\sum_{m=0}^{n}\frac{\Gamma(m+1/2)}{(m+1/2)^2 m!}\right).$$
Now, for the integral term in~\eqref{eq:Z12_Burger},%
\begin{align*}
        \left|\int_0^{\infty}\bar{u}^2\ih\partial_x(\hat{\psi}_m e^{x^2/4})\mathrm{d} x\right|&=\left|\int_0^{\infty}\bar{u}^2\ih\left(\partial_x\hat{\psi}_m +\frac{x}{2}\hat{\psi}_m\right) e^{x^2/4}\mathrm{d} x\right|\\
        &\leq \left|
        \left\langle \ih, \bar{u}^2\left(\partial_x\hat{\psi}_m +\frac{x}{2}\hat{\psi}_m\right) \right\rangle\right|\\
        &\leq \left|
        \left\langle \ih, \OP\left(\bar{u}^2\left(\partial_x\hat{\psi}_m +\frac{x}{2}\hat{\psi}_m\right)\right) \right\rangle\right|\qquad \mbox{as $\ih\in \iH$}\\
        & \leq\|\ih\|_{L^2(\mu)}\left\|\OP\left(\bar{u}^2\left(\partial_x\hat{\psi}_m +\frac{x}{2}\hat{\psi}_m\right)\right) \right\|_{L^2(\mu)}\\
        &\leq\frac{\|\ih\|_{H^2(\mu)}}{\lambda_{n+1}}\left(\left\|\bar{u}^2\left(\partial_x\hat{\psi}_m +\frac{x}{2}\hat{\psi}_m\right)\right\|_{L^2(\mu)}^2- \left\|P_{n}\left(\bar{u}^2\left(\partial_x\hat{\psi}_m +\frac{x}{2}\hat{\psi}_m\right)\right) \right\|_{L^2(\mu)}^2\right)^{1/2}.
\end{align*}
Thus, with
$$\tilde{w}_m = \bar{u}(0)^2\hat{\psi}_m(0)\sqrt{S_n}+\frac{1}{\lambda_{n+1}}\left(\left\|\bar{u}^2\left(\partial_x\hat{\psi}_m +\frac{x}{2}\hat{\psi}_m\right)\right\|_{L^2(\mu)}^2- \left\|P_{n}\left(\bar{u}^2\left(\partial_x\hat{\psi}_m +\frac{x}{2}\hat{\psi}_m\right)\right) \right\|_{L^2(\mu)}^2\right)^{1/2},$$
we can choose
$$Z^{12} = \left\|\,\left|\mathcal{L}A_n \mathcal{L}^{-1}\right|\tilde{w}\right\|_{L^2(\mu)}.$$
An estimate based on Parseval's identity on $H^1(\mu)$ is also possible, but gives similar results.

\begin{rmk}
Note that if $u(x) = q(x)e^{-x^2/4}$ for a polynomial $q$, then $\partial_x u(x) = (\partial_x q(x) -x/2q(x))e^{-x^2/4}$ and thus the integrals necessary to compute our bounds can also be evaluated with quadrature rules as in Section~\ref{sec:quadrature}.
Concerning $L^{\infty}$-bounds, from~\cite[\href{https://dlmf.nist.gov/18.14.E9}{(18.14.9)}]{NIST:DLMF}, we have that for all $m\in\N$
\begin{align*}
    \|\hat{\psi}_m\|_{\infty}\leq\|\hat{\psi}_m e^{x^2/8}\|_{\infty} = \|e^{-x^2/8}\hat{H}_{2m}(x/2)\|_{\infty}\leq \pi^{-1/4},
    \end{align*}
    and for all $m\in\N^*$
    \begin{align*}\|\partial_x\hat{\psi}_m\|_{\infty}\leq \left\|\sqrt{m}e^{-x^2/4}\hat{H}_{2m-1}(x/2) - \frac{x}{2}\hat{\psi}_m\right\|_{\infty}\leq \pi^{-1/4}\left(\sqrt{m}+\frac{1}{\sqrt{e}}\right)\qquad \mbox{as $\left\|\dfrac{x}{2}e^{-x^2/8}\right\|_{\infty}=\dfrac{1}{\sqrt{e}}$}.
\end{align*}
\end{rmk}
With $n=1500$, we find the approximate solution $\bar{u}$ represented in Figure~\ref{fig:burger}.
\begin{figure}[H]
    \centering
    \includegraphics[width=0.6\textwidth]{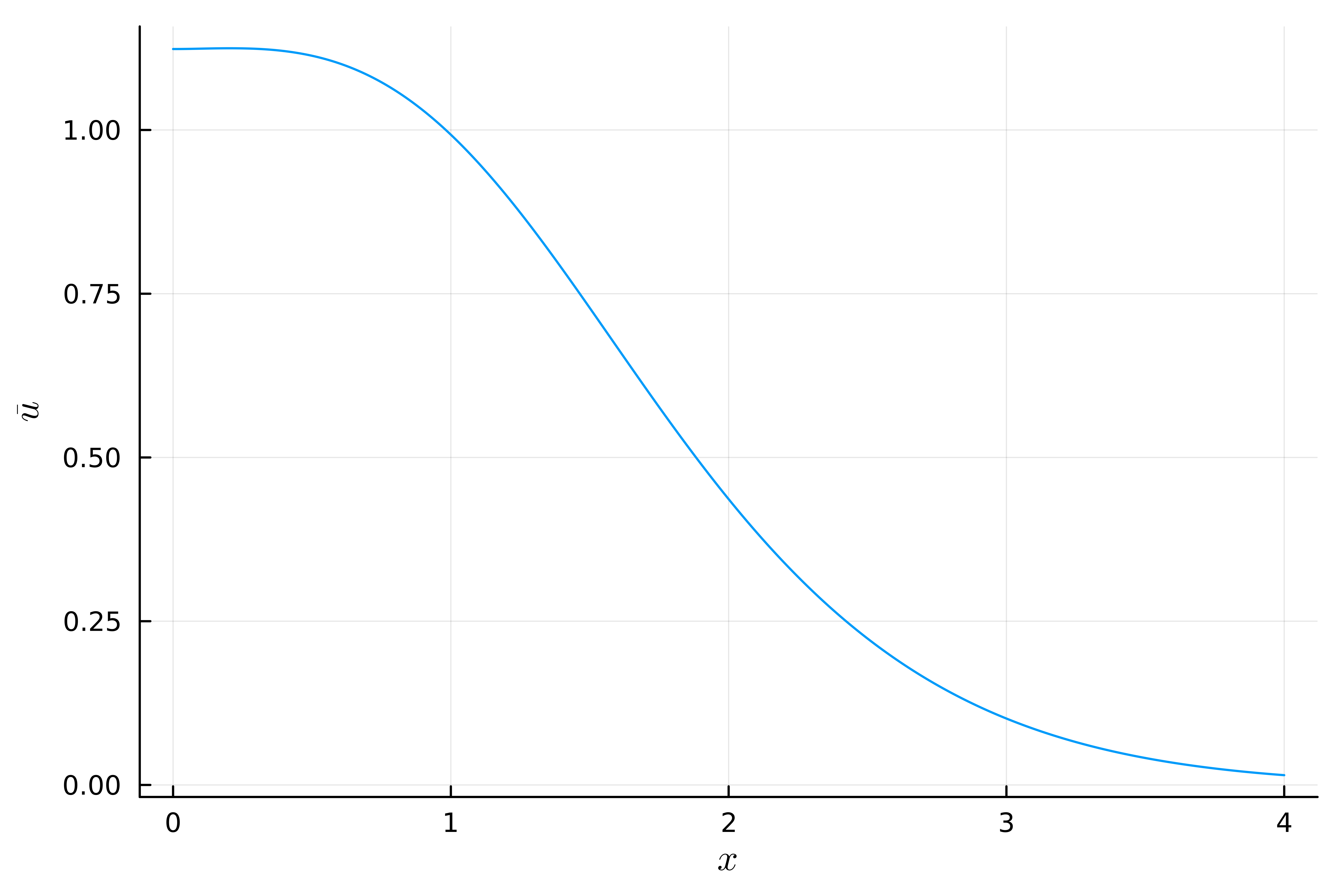}
    \caption{The numerical solution $\bar{u}$ to Eq.~\eqref{eqn:burger3}.\label{fig:burger}}
\end{figure}

\begin{thm}
\label{thm:Burgers}
    Let $\bar{u}:\R_+\to\R$ be the function represented on the interval $[0, 4]$ in Fig~\ref{fig:burger} and whose description in terms of ``Fourier'' coefficients in $\{\hat{\psi}_m\}_{m=0}^n$ is available at~\cite{Chu2024CodeOn}, then there exists a positive solution $u^{\star}\in\mathcal{H}$ to Eq.~\eqref{eqn:burger3} such that
    $$\|u^{\star} - \bar{u}\|_{H^2(\mu)} \leq 10^{-3}.$$
\end{thm}

\begin{proof}
    We consider $F :\mathcal{H}\to\mathcal{H}$ defined in~\eqref{eqn:F-burger}. In the file \texttt{Burger/proof.ipynb} at \cite{Chu2024CodeOn}, we construct $A_n$ (and thus indirectly $A$ and $T$) as outlined in Section~\ref{sec:fixedpoint} and following the estimates above, we find the following sufficient bounds:
    \begin{align*}
        Y &= 0.00075636391,\\
        Z_1 &= 0.065135932,\\
        Z_2 &= 343.3917,\\
        Z_3 &= 556.478,\\
        \bar{\delta} &= 0.00271646316,
    \end{align*}
    and we conclude with Corollary~\ref{cor:fixed-point} to the existence of a solution $u^{\star}\in\mathcal{H}$ to Eq.~\eqref{eqn:burger3}. As Eq.~\eqref{eqn:burger3} can be seen as an ordinary differential equation with a smooth vector field, we immediately get that $u^{\star}\in \mathcal{C}^{\infty}(\R_+)$ and thus we can apply Lemma~\ref{lem:max-princ} (again to $\varphi = e^{|x|^2/8}u^{\star}$) with $b(x) = u^{\star 2}(x)$ and $c(x)= x^2/16 -xu^{\star 2}(x)/4$, to get that $u^{\star}$ is positive.
\end{proof}

\section*{Acknowledgments}

The authors thank Miguel Escobedo, Olivier Hénot, Marc Nualart Batalla and Sheehan Olver for useful discussions.

This work was conducted as part of HC's doctoral research at Imperial College, supported by a scholarship from the Imperial College London EPSRC DTP in Mathematical Sciences (EP/W523872/1) and by the EPSRC Centre for Doctoral Training in Mathematics of Random Systems: Analysis, Modelling and Simulation (EP/S023925/1).
MB was supported by the ANR project CAPPS: ANR-23-CE40-0004-01, and also acknowledges the hospitality of the Department of Mathematics of Imperial College London. MB and HC are also grateful to the CNRS--Imperial \emph{Abraham de Moivre} IRL for supporting their research.

\printbibliography[heading=bibintoc]

\end{document}